\documentclass[reqno]{amsart}
\usepackage[T1]{fontenc}
\usepackage{palatino}

\usepackage{fix-cm}
\usepackage[margin = 1.25in]{geometry}
\usepackage[fontsize=11pt]{fontsize}

\usepackage[
    colorlinks=true,       
    linkcolor=Maroon,        
    citecolor=OliveGreen,        
    filecolor=magenta,     
    urlcolor=Turquoise,         
    pdfborder={0 0 0},     
    breaklinks=true,       
    pdfpagemode=UseNone,   
    pdfencoding=auto,      
    unicode=true           
]{hyperref}

\usepackage[
  backend=biber,
  style=alphabetic,
  url=true,
  eprint=true,
  isbn=false,
  doi=false,
]{biblatex}

\usepackage[all,arc]{xy}
\usepackage{enumerate}
\usepackage{blindtext}
\usepackage{amsmath}
\usepackage{amssymb}
\usepackage{amsfonts}
\usepackage{amsthm}
\usepackage{physics}
\usepackage{mathtools}
\usepackage{indentfirst}
\usepackage{tikz-cd}
\usepackage{tikz}
\usepackage{lipsum}
\usepackage{ytableau}
\usepackage{graphicx}
\usepackage{subfigure}
\usepackage{caption}
\usepackage{textcomp}
\usepackage{url}
\usepackage{colonequals}
\usepackage{comment}
\usepackage{enumitem}
\usepackage{dsfont}
\usepackage{bbm}
\usepackage{graphicx}
\usepackage{multicol}
\usepackage{bbold}
\usepackage{enumitem}
\usepackage{framed}
\usepackage[dvipsnames]{xcolor}
\usepackage{tcolorbox}
\usepackage{etoolbox}

\makeatletter
\patchcmd{\@settitle}
  {\bfseries}
  {\fontsize{11.5}{11.5}\fontfamily{qcs}\selectfont\bfseries}
  {}{}

\patchcmd{\@setauthors}
    {\centering\footnotesize}
  {\centering\fontsize{11}{11}\fontfamily{qcs}\selectfont}
  {}{}

\patchcmd{\@setauthors}
  {\MakeUppercase{\authors}}
  {\authors}
  {}{}

\newcommand{\titleaddress}[1]{\gdef\@titleaddress{#1}}
\newcommand{\titleemail}[1]{\gdef\@titleemail{#1}}
\newcommand{\titledate}[1]{\gdef\@titledate{#1}}

\let\@titleaddress\@empty
\let\@titleemail\@empty
\let\@titledate\@empty

\patchcmd{\@setauthors}
  {\endtrivlist}
  {%
    \par\medskip
    {\normalfont\fontsize{9}{11}\selectfont
      \ifx\@titleaddress\@empty\else {\itshape\@titleaddress}\par\fi
      \ifx\@titleemail\@empty\else {\fontfamily{pcr}\@titleemail}\par\fi
      \ifx\@titledate\@empty\else \@titledate\par\fi
    }%
    \endtrivlist
  }
  {}{}

\makeatother

\title{Singular-value gap of nonreversible Markov processes}
\author{Ruochuan Xu}
\titleaddress{Stanford University}
\titleemail{ruochuan@stanford.edu}
\titledate{May 2026}

\makeatletter
\def\l@subsection{\@tocline{2}{0pt}{30pt}{5pc}{}}
\makeatother

\begin{document}

\begin{abstract} 
We consider a generalization of the spectral gap of reversible Markov generators
to nonreversible processes,
following the recent work \cite{chatterjee} on nonreversible finite-state Markov chains. 
Extending Chatterjee's observations, we find that this spectral quantity that we call the
\textit{singular-value gap} characterizes the convergence of empirical averages,
providing upper and lower bounds for finite-time variance uniformly 
over $L^2$-functions. A key observation is that when the singular-value gap is positive,
the generator is invertible on the $L^2$-orthogonal complement of constant functions.
In particular, the Poisson equation $-Lf = g$ can be solved, which enables our proof and
connects our results to asymptotic variance and associated central limit theorems.

We also compare the singular-value gap with the spectral gap of the reversibilized process, the
mixing time in total-variation distance, and the Cheeger constant. Several examples are provided 
throughout the text. Among other potential applications 
of the singular-value gap, these examples illustrate that a positive singular-value gap
can help with variance reduction for observable classes in MCMC sampling, uncover
slow-mixing mechanisms, and certify convergence of empirical averages for 
diffusion operators with complicated spectrum.
\end{abstract}

\maketitle

\newtheorem{definition}{Definition}
\newtheorem{theorem}{Theorem}
\newtheorem{corollary}{Corollary}
\newtheorem{proposition}{Proposition}
\newtheorem{lemma}{Lemma}
\newtheorem{remark}{Remark}
\newtheorem{example}{Example}

\newcommand{\E}{\mathds{E}}
\newcommand{\Var}{\operatorname{Var}}
\newcommand{\Cov}{\operatorname{Cov}}
\newcommand{\Dom}{\operatorname{Dom}}
\newcommand{\pathavg}{\mathrm{path}}
\newcommand{\inner}[2]{\left\langle #1,#2\right\rangle}

\newpage
\tableofcontents
\newpage

\section{Introduction}

\subsection{Setting and motivation}
Let $(X_t)_{t\ge 0}$ be a time-homogeneous Markov process with an invariant probability measure $\mu$ on a measurable
space $S$.
The associated semigroup $(P_t)_{t\geq 0}$ on $L^2(\mu)$ is given by $P_t g\,(x) = \E_x[g(X_t)]$, where $\E_x$ denotes the expectation conditioned on $X_0 = x$. Assume that $(P_t)_{t\geq 0}$ is strongly continuous, with the generator 
$L$ defined on $\Dom(L)$ by $L g = \lim_{t \downarrow 0} \frac{P_t g - g}{t}$, where 
$\Dom(L)$ consists of $L^2$-functions $g$ for which the limit exists. By the Hille-Yosida theorem, 
$L$ is closed and $\Dom(L)$ is dense in $L^2(\mu)$.
We sometimes consider discrete-time Markov chains as well, in which case the setting will be clarified.
Let 
\[
\langle f, g\rangle \coloneqq \int \bar{f} g \, d\mu
\]
denote the $L^2(\mu)$ inner product, and 
$\| \cdot \|$ the corresponding norm.
The empirical average $\mu_t g$ defined below can be understood as 
a Bochner integral, since $s \mapsto g(X_s)$ is continuous as 
an $L^2(\mathds{P}_\mu)$-valued function by the strong continuity of $(P_t)_{t \geq 0}$.
If the process has jointly measurable sample paths, this coincides with the usual pathwise time integral; for simplicity, we assume joint-measurability. 
We shall also restrict attention to centered observables $g$ in $L^2_0(\mu):=\left\{g\in L^2(\mu): \int g\,d\mu=0\right\}$. 
Note that for $g$ in $L^2(\mu)$, we can always take $g - \mu g \in L^2_0(\mu)$.

A basic problem in the study of a Markov process
is to quantify the speed with which the process converges. Two natural but distinct questions are: how quickly the time-$t$ law approaches $\mu$, and how quickly the \textit{empirical average}
\[
    \mu_t g:=\frac1t\int_0^t g(X_s)\,ds
\]
approaches $\mu g \coloneqq \int g \, d\mu$, the equilibrium average.

If the Markov process is started from stationarity with $X_0 \sim \mu$,
the convergence of empirical averages is naturally measured by
the \textit{empirical variance} 
\[
\Var_\mu(\mu_t g) \coloneqq \mathds{E}_\mu |\mu_t g - \mu g|^2.
\]
In the context of Markov chain Monte Carlo (MCMC) algorithms, the two questions about convergence raised above 
target, respectively, the burn-in time and the time required for accurate sampling after
(approximate) stationarity is reached; see e.g. \cite{HobertJonesHonest} \cite{RobertsRosenthal2004}.

In the present paper, we study the same basic convergence questions 
but depart from the traditional perspectives in some respects.
There is already substantial research on nonreversible Markov chains and processes--see the review 
in section 1.9 of \cite{chatterjee} and section \ref{related literature section} below. The interest in such 
processes has been mainly driven by their accelerated convergence.
While most previous studies focus on specific classes of examples, we follow
the work on nonreversible finite-state Markov chains in \cite{chatterjee} and consider 
a universal spectral theory for nonreversible 
Markov processes on arbitrary state spaces. Even without reversibility, we find that a generalization of the spectral gap that we call the singular-value gap remains useful and informative. The main finding is that a
positive singular-value gap gives finite-time, observable-independent bounds on empirical variance and
naturally relates to other quantities for Markov process convergence.

\subsubsection{Spectral gap vs. singular-value gap}
In the reversible setting, different notions of convergence are tied to the same spectral quantity. 
Since $-L$ is nonnegative and self-adjoint on $L^2(\mu)$, its spectrum is contained in $\mathds{R}_{\geq 0}$.
The \textit{spectral gap} $\lambda$ of $L$ is then defined as 
\[
    \lambda \coloneqq \inf \sigma(-L\mid_{L^2_0(\mu)}),
\]
where $\sigma$ denotes the spectrum. For a discrete-time finite-state Markov chain with transition matrix $P$,
the generator is given by
$L = P - I$, in which case $\lambda$ is usually characterized as one minus the second-largest eigenvalue of $P$.
When $\lambda > 0$, we have 
exponential $L^2$-contraction of the semigroup: $\|P_t f\| \leq e^{-\lambda t} \|f\|$ for all centered observables $f$. 
This leads to the Poincaré inequality $\|f\|^2 \le \lambda^{-1}\langle f,-Lf\rangle$,
where $\langle-Lf,f\rangle$ is the Dirichlet form evaluated on $f$ and is
often written as $\mathcal{E}(f,f)$. In fact, the spectral gap can be characterized as (the inverse of)
the optimal constant in the Poincaré inequality:
\begin{align} \label{reversible spectral gap characterization}
    \lambda = \inf_{f\in \Dom(L)\cap L^2_0(\mu)\backslash \{0\}} 
    \frac{\langle-Lf,f\rangle}{\|f\|^2}.
\end{align}
Although less emphasized in the continuous-time literature, it is a quick consequence of 
$L^2$-contraction and the standard identity \eqref{variance of empirical average 2} that when
the process is reversible,
\begin{align} \label{reversible empirical variance bound}
        \operatorname{Var}_\mu ( \mu_t g)
        \le
        \frac{2}{\lambda t} \|g\|^2 
\end{align}
for all $g \in L^2_0(\mu)$.
For Markov chains on finite state space, 
the time-$n$ empirical average is
\[
 \mu_n g \coloneqq \frac{1}{n} \sum_{i=0}^{n-1} g(X_i),
\]
and the corresponding empirical variance bound holds exactly as in \eqref{reversible empirical variance bound}
for a reversible chain,
with continuous $t$ replaced by discrete $n$; see Lemma 12.22 in \cite{levin-peres}.
This perspective of the spectral gap as a canonical convergence quantity 
is central in the classical theory of reversible Markov chains and Markov diffusion processes; see
\cite{levin-peres} \cite{bakry2014analysis} and references therein.

For nonreversible processes, the situation is more delicate, as the generator is no longer self-adjoint. 
There are multiple ways to retrieve a spectral theory:
one may still consider the spectral gap $\gamma_s$ of the symmetrized process whose generator is given by 
\begin{align} \label{symmetrized generator}
    L_s \coloneqq \frac{L + L^*}{2};
\end{align}
see the more precise definition of $\gamma_s$ in \eqref{definition of gamma_s through symmetric form}
below.
Alternatively, one may also consider the real-part spectral gap of \(L\)
or resolvent and pseudospectral quantities. There is no longer a single canonical spectral object.

For nonreversible finite-state Markov chains, it is demonstrated in \cite{chatterjee} 
that empirical variance of observables can be uniformly bounded by
the singular-value gap $\gamma$, defined as the second-smallest singular value of $I-P$,
where $P$ is the transition matrix.
Note that $L = P-I$ is the generator in discrete time, and we may equivalently write
\begin{align} \label{gap definition discrete time}
    \gamma:=\inf_{f\in  L^2_0(\mu)\backslash \{0\}}  \frac{\norm{(P-I)f}}{\norm{f}}.
\end{align}
With the formulation in 
\eqref{gap definition discrete time}, it is natural to consider the continuous-time analog
\begin{align} \label{gap definition continuous time}
\gamma:=\inf_{f\in \Dom(L)\cap L^2_0(\mu)\backslash \{0\}}  \frac{\norm{Lf}}{\norm{f}}.
\end{align}
Indeed, this \textit{singular-value gap} has also been studied in the context of
continuous-time Markov processes,
most recently in \cite{eberle-lorler}\cite{huang2025bernsteintypeinequalitiesmarkovchains}. 
The terminology ``singular-value gap'' is literal.  
Let $L_0 \coloneqq L \vert_{L^2_0(\mu)}$, with $\Dom(L_0) = \Dom(L) \cap L^2_0(\mu)$.
Consider the nonnegative
self-adjoint operator $|L_0|:=(L_0^*L_0)^{1/2}$.
By the polar decomposition for closed densely defined operators,
$\|L_0 f\|=\||L_0|f\|$ for $ f\in\operatorname{Dom}(L_0)$.
Therefore,
\[
        \gamma
        =
        \inf_{\substack{f\in \operatorname{Dom}(L_0), \ \|f\|=1}}
        \|L_0f\|
        =
        \inf \sigma(|L_0| ).
\]
Equivalently,
\[
        \gamma^2=\inf\sigma(L_0^*L_0).
\]
Thus, in a finite state space, $\gamma$ is exactly the second-smallest singular value of the
matrix \(-L\). When $L$ is normal, by the spectral theorem, 
we also have
\[
\gamma = \mathrm{dist}\big(0, \sigma(L_0)\big),
\]
but this interpretation can fail without normality.

As we shall show,
among various inequivalent metrics of 
convergence for a nonreversible Markov process,
the singular-value gap
$\gamma$ defined in \eqref{gap definition continuous time} 
specifically characterizes the convergence for empirical 
averages. This can be intuitively understood through the following simple example, which
distinguishes $\gamma$ from the spectral gap $\gamma_s$ of the symmetric part of the generator.

\begin{example} \label{first example BM with drift}
    Consider Brownian motion with constant drift on the unit circle $\mathds T  = \mathds R/ 2 \pi \mathds Z$:
    \[
    dX_t = \alpha \, dt+ \sqrt{2} \, dB_t \quad \mod 2\pi, \quad \text{with generator} \quad
    L_\alpha = \partial_\theta^2 + \alpha \partial_\theta.
    \]
    The invariant measure $\mu$ is normalized Lebesgue measure on $\mathds T$. An orthonormal 
    basis for $L^2(\mu)$ is given by the Fourier modes $e_k(\theta) = e^{ik \theta}$ for $k \in \mathds{Z}$, where 
    $(e_k)_{k \neq 0}$ form a basis for $L^2_0(\mu)$. Each $e_k$ is an eigenfunction 
    for $-L_\alpha$ with eigenvalue $k^2 - i\alpha k$, so we get
    \[
    \gamma = \inf_{k \neq 0} | k^2 - i\alpha k| = \sqrt{1+\alpha^2}.
    \]
    On the other hand, the symmetric part of $L_\alpha$ is $L_s = \partial_\theta^2$, and the symmetrized 
    gap is 
    \[
    \gamma_s = \inf_{k \neq 0} k^2 = 1.
    \]
    One may also check that $P_t e_k = e^{-(k^2 - i\alpha k)t}e_k$, hence
    $\|P_t\|_{L^2_0 \to L^2_0} = e^{-t}$. The $L^2$-contraction rate is correctly captured by $\gamma_s$. 
    In fact, it is in general true that positive $\gamma_s$ implies exponential contraction of $P_t$; see \eqref{L2 contraction by gamma_s} below.
     However, one may check $\lim_{t \to \infty} t \Var_\mu(\mu_t e_k) = \frac{2}{k^2 + \alpha^2}$, so
     \[
     \sup_{g \in L^2_0(\mu): \|g\| = 1} \, \lim_{t \to \infty } t \Var_\mu(\mu_t g) = \frac{2}{1+ \alpha^2}
     \asymp\frac{1}{\gamma^2}.
     \]
     The convergence speed of empirical averages is captured by $\gamma$ but missed by $\gamma_s$. 
\end{example}
The observations above are consistent with the intuition that a 
drift of strength $\alpha$ should accelerate the convergence of empirical averages, but it merely translates
the law of $X_t$ by $\alpha t$ and does not help with its convergence to the uniform distribution.

\subsubsection{Asymptotic variance vs. empirical variance} 
\label{asymptotic variance subsubsection}
Since $\Var_\mu (\mu_t g)$ often decays as $t^{-1}$, it is of interest to study
the \textit{asymptotic variance}, which is naturally connected to functional CLT results.
For \(g\in L_0^2(\mu)\), the asymptotic variance of $g$, when it exists, is
defined by
\begin{align} \label{asymptotic variance definition}
        \sigma_g^2
        :=
        \lim_{t\to\infty} t\,\operatorname{Var}_\mu(\mu_tg).
\end{align}
We will call $\sigma_g \geq 0$ the \textit{asymptotic standard deviation} of $g$ governing
long-time fluctuation:
\[
        \sqrt{\Var_\mu (\mu_t g)} \sim \frac{\sigma_g}{\sqrt{t}}.
\]
A standard way to analyze $\sigma_g^2$ is through the 
Poisson equation. If we can find $f \in \Dom(L)$ that solves $-Lf = g$, then
\begin{align} \label{asymptotic variance representation with L inverse}
        \sigma_g^2
        =
        2\operatorname{Re}\langle g,f\rangle .
\end{align}
See Lemma \ref{appendix lemma proof of asymptotic variance formula} for a 
short derivation of $\eqref{asymptotic variance representation with L inverse}$.
When we can solve the Poisson equation for $g$ and the process satisfies a few 
other technical assumptions,
Dynkin's formula 
and central limit theorem for martingales imply that
$\sqrt{t} \, \mu_t g$ converges to a Gaussian with mean zero and variance $\sigma_g^2$.
This is the classical CLT result of \cite{bhattacharya1982functional}.
A substantial theory of functional CLT, notably initiated by \cite{Kipnis-Varadhan1986}, establishes
analogous results without requiring a solution to the Poisson equation. See Chapter 2 of 
\cite{komorowski2012fluctuations} and the references therein.

For the purposes of this paper, given a positive $\gamma$, we are in the simpler regime that
the Poisson equation is solvable. In fact, $L^{-1}$ exists on $L^2_0(\mu)$ with 
$\|L^{-1}\|_{L^2_0 \to L^2_0} \leq \gamma^{-1}$,
thanks to Lemma \ref{main invertible lemma}. Therefore, the Poisson equation is solvable for all functions in
\(L^2_0(\mu)\). This yields a CLT under extra technical assumptions;
for instance, when the process is progressively measurable, by
Theorem 2.1 in \cite{bhattacharya1982functional}. Even if we only know that
$\gamma > 0$, we have a simple bound for the asymptotic variance:
\begin{align} \label{upper bound on asymptotic variance by global gamma}
        \sigma_g^2
        =
        2\operatorname{Re}\langle g,-L^{-1}g\rangle \leq \frac{2}{\gamma} \|g\|^2.
\end{align}
In the above bound, the exact value of \(\sigma_g^2\) may be much smaller and dependent
on the particular observable \(g\).
The singular-value gap studied
in this paper provides observable-class-level bounds on finite-time empirical
variance and asymptotic variance.
The asymptotic variance is usually studied in the literature to obtain sharp long-time 
behavior for specific observables. Knowledge of asymptotic variance
helps refine bounds on finite-time empirical variance:
see Theorem \ref{finite time variance bound by asymptotic variance} below.

\subsection{Main results}
We shall state our main results in a subspace-dependent form. Let 
$V \subset L^2_0(\mu)$ be a subspace. Consider the singular-value gap 
of $L$ on $V$ given by 
\begin{align} \label{gap definition subspace}
    \gamma_V \coloneqq \inf_{f \in \Dom(L) \cap V \backslash \{0\}} \frac{\|Lf\|}{\|f\|}.
\end{align}
Let $U \coloneqq L(V \cap \Dom(L))$. Note that
since $\mu$ is invariant, $\int Lf\,d\mu=0$ for $f \in \Dom(L) \cap L^2_0 (\mu)$; hence $U \subset L^2_0(\mu)$.
We quantify the worst-case convergence of empirical averages of 
observables in $U$ by the quantity
\begin{align} \label{definition of empirical Delta_t^U}
  \Delta_t^U
:=\sup_{g \in U:\, \|g\| = 1}
\sqrt{\Var_\mu (\mu_t g)} = \sup_{g \in U+\mathds{C} \mathbf{1}: \, \|g - \mu g\| = 1} \sqrt{\Var_\mu(\mu_t g)}
\end{align}
We have the following upper and lower bounds on $\Delta_t^U$ in terms of $\gamma_V$:
\begin{theorem} \label{main subspace dependent theorem}
    In the above setting, assume $\gamma_V > 0$. Then 
    \begin{align*}
        &\text{(i)} \qquad \Delta_t^U \leq \frac{2}{\sqrt{t \mkern1mu \gamma_V }} \\
        &\text{(ii)}\sup_{s \in [t, 2t]}  \Delta_s^{U+V} \geq \frac{1}{\frac{3}{2}t \mkern1mu \gamma_V + 3}
    \end{align*}
\end{theorem}

With the above theorem and the key input of Lemma \ref{main invertible lemma}, we obtain control of the empirical standard deviation of all observables in $L^2(\mu)$, quantified by 
\begin{align} \label{definition of empirical Delta_t}
  \Delta_t
:=\sup_{g \in L^2_0(\mu):\, \|g\| = 1}
\sqrt{\Var_\mu (\mu_t g)} = \sup_{g \in L^2(\mu): \, \|g - \mu g\| = 1} \sqrt{\Var_\mu(\mu_t g)}
\end{align}
\begin{corollary} \label{main global corollary}
    Assume the global singular-value gap $\gamma$ defined in \eqref{gap definition continuous time} 
    satisfies $\gamma > 0$. Then 
    \begin{align*}
        &\text{(i)} \qquad \Delta_t \leq \frac{2}{\sqrt{\gamma \mkern1mu t}} \\
        &\text{(ii)} \sup_{s \in [t, 2t]}  \Delta_s \geq \frac{1}{\frac{3}{2} \gamma \mkern1mu t  + 3}
    \end{align*}
\end{corollary}

Note that for a pure rotation on the unit circle, for instance, the empirical average is exactly the equilibrium average at times when an integer number of rotations are performed. Therefore,
for a general nonreversible dynamics, it seems infeasible 
to obtain lower bounds for empirical variance without taking the worst case over a time interval.

One might also wonder whether the bounds above are optimal. We discuss simple examples in 
section \ref{sharpness of bounds subsection} that show that the upper and lower bounds in 
Corollary \ref{main global corollary} are generally sharp up to constants. However, Theorem 
\ref{finite time variance bound by asymptotic variance} below gives the following 
refinement in terms of the asymptotic variance:
\begin{align} \label{refined empirical variance bound in main results}
\frac{\sigma_g^2}{t} - \frac{4}{\gamma^2 t^2} \|g\|^2   \leq \Var_\mu(\mu_t g) \leq \frac{\sigma_g^2}{t} + 
\frac{4}{\gamma^2 t^2} \|g \|^2
\end{align}
In particular, when 
$\sigma_g^2 > 0$ for observables $g$ under consideration, the standard deviation 
of empirical averages of $g$ decays as $t^{-1/2}$, the expected CLT scale.
For certain dynamics and observables, the empirical 
standard deviation can decay faster on the order of $t^{-1}$; we discuss a simple instance in Example 
\ref{pure rotation on circle with alpha drift}. The bound
\eqref{refined empirical variance bound in main results} shows that such a faster 
convergence happens exactly when the asymptotic variance vanishes,
but we do not explore the separation between 
the $t^{-1/2}$ and $t^{-1}$ regimes in more detail.

\subsubsection{Outline} We contextualize our perspective within the existing literature
in section \ref{related literature section}. In section \ref{main results section}, we 
prove the local Theorem \ref{main subspace dependent theorem} and deduce the global Corollary
\ref{main global corollary} from Lemma \ref{main invertible lemma}, which gives a bounded 
inverse generator when $\gamma > 0$. Using this lemma, we also extend Chatterjee's original 
theorem for finite-state Markov chains to arbitrary state spaces, confirming that
the singular-value gap $\gamma$ is a universal quantity for convergence. 
We then discuss the sharpness of our variance bounds and prove the
sharpened version \eqref{refined empirical variance bound in main results} in 
terms of asymptotic variance. 

In section \ref{OU process section}, we study in detail the example of a two-dimensional OU process with 
rotational drift. This example is inspired by \cite{DLP2016} and sits within
a substantial literature on accelerating the convergence of diffusion processes
by adding an irreversible drift. The main point of this example is that irreversibility may
accelerate the convergence of empirical averages only for certain subspaces of observables.

We explore the connections between $\gamma$ and some other common quantities related to mixing
in section \ref{relation to other notions section}.  
In section \ref{relation to symmetrized process subsection}, we bound $\gamma$ in terms of 
$\gamma_s$, the spectral gap of the symmetrized process.  
Theorem \ref{mixing time lower bound} shows that $\gamma^{-1}$ provides 
a lower bound for the mixing time in total-variation distance, and Proposition 
\ref{mixing time upper bound finite state space} gives an upper bound on the mixing time 
in the restricted setting of jump processes on finite state space.
We then connect the singular-value gap to the Cheeger constant, also known as the bottleneck ratio.
Theorem \ref{Cheeger theorem} extends the previous results of \cite{lawler-sokal}\cite{chatterjee}
to nonreversible jump processes. 

Our theory can be applied under minimal assumptions, but we only survey two more examples 
in section \ref{applications section}. We consider random walks on binary de Bruijn graphs
to show that, as $\gamma^{-1}$ lower bounds the total-variation mixing time
by Theorem \ref{mixing time lower bound}, the 
singular-value gap could more sharply record mixing properties compared to some other spectral quantities. We then observe that a compact resolvent criterion in 
Proposition \ref{compact resolvent criterion proposition} is enough
to show that  $\gamma > 0$ for certain Markov diffusion generators, without explicit knowledge of the value of $\gamma$. We study a specific hypoelliptic generator 
on the 2-torus and show that $\gamma > 0$, hence
empirical averages converge with fluctuations of order at most $t^{-1/2}$. 
This example provides an interesting bridge between abstract PDE techniques 
and concrete sampling variance estimates.
It would be interesting to find more substantial applications of the singular-value gap beyond the examples given in this paper.

\subsection{Related literature} \label{related literature section}

As mentioned above, the closest predecessor to the present work is the
theory of nonreversible spectral gaps developed in \cite{chatterjee}
for nonreversible finite-state Markov chains. Chatterjee shows that
for a chain with transition matrix $P$, the
reciprocal of the gap $\gamma$ defined as the second-smallest
singular value of $I-P$ characterizes
the time scale for the convergence of empirical averages. He also relates this
quantity to mixing time, Cheeger constants, and path arguments. Our work
follows this generator-singular-value viewpoint but extends it to
continuous time, general state space, and localized observable classes.

An earlier precursor is \cite{Dinwoodie1998}, which obtains deviation bounds for the empirical
averages of nonreversible chains in terms of the singular-value gap of the discrete
generator. More recently, \cite{eberle-lorler} discusses the
singular-value gap for nonreversible diffusion lifts, although the main
relaxation notion there is an \(L^2\)-contraction time of the
semigroup rather than empirical-average variance. These generator singular
values should be distinguished from other nonreversible spectral quantities.
Notably, for discrete-time chains,
Fill's multiplicative reversibilization \cite{Fill1991} and Paulin's pseudo-spectral gap 
\cite{paulin2018concentration} use
singular values of transition operators, or of products such as
\((P^*)^kP^k\), to control convergence and concentration. 
\cite{KonMey2011} emphasize yet another perspective: for general state-space
nonreversible chains, spectral theory for geometric ergodicity is often
better formulated in weighted \(L^\infty\)-spaces.
In the present paper, we keep our focus on \(L^2(\mu)\).

There is a large literature showing that nonreversibility can improve
convergence, hence facilitate sampling. In finite state space, lifting is
a central mechanism. \cite{liftingchen} shows that lifting can speed up
mixing and that square-root-type speedups are optimal in broad settings. 
\cite{DiaconisNonreversibleSampler} analyzes a concrete lifted
nonreversible sampler.
Analogous bounds for lifting constructions
on continuous state space are developed in \cite{RamananSmith2018}.

For diffusions, \cite{hwang1993accelerating}\cite{HAcceleratingDiffusions} show
that adding a
divergence-free drift preserving a target distribution can accelerate
convergence. \cite{LNP2013} studies optimal nonreversible
linear drifts, emphasizing both exponential rates and
prefactors.
\cite{HWANG2015} gives general
comparison results showing that antisymmetric perturbations reduce
asymptotic variance for diffusions. There are also more detailed investigations of
variance reduction:
\cite{Rey_Bellet-Spiliopoulos_2015} uses large deviations of the
empirical measure to quantify variance reduction for irreversible Langevin
samplers, while \cite{DLP2016} analyzes variance reduction via the Poisson
equation and calculations for linear and quadratic observables. Our results are complementary:
rather than considering irreversible perturbations of a certain form or 
reducing observable-specific asymptotic variance, we ask what can be concluded from a
positive singular-value gap of the generator, globally or on an observable
subspace.
The connections of our perspective with asymptotic variance and central-limit theorems are discussed in section \ref{asymptotic variance subsubsection} above.

There is also substantial literature on finite-time deviation bounds for
empirical averages. In the reversible setting, variance bounds
for stationary empirical averages follow from the spectral gap, as recorded in 
\eqref{reversible empirical variance bound}.
\cite{Lezaud2001} proves Chernoff and Berry--Esséen
inequalities for empirical means of general ergodic Markov processes with a
spectral gap. \cite{paulin2018concentration} obtains concentration, 
variance and Bernstein-type
bounds for possibly nonreversible Markov chains through the
pseudo-spectral gap. \cite{JoulinOllivier_2010} derives 
nonasymptotic empirical-average error estimates
under a positive-curvature
assumption.
More recently, Bernstein-type inequalities
are obtained in \cite{huang2025bernsteintypeinequalitiesmarkovchains}, using 
the singular-value gap \eqref{gap definition discrete time}, which the authors call
the iterated Poincaré gap. In continuous time, \cite{huang2025bernsteintypeinequalitiesmarkovchains} uses a slightly different 
characterization of the singular-value gap \eqref{gap definition continuous time}
and proves concentration inequalities for empirical averages
of bounded observables, 
allowing nonstationary starts.

In contrast to some aforementioned studies with MCMC sampling in mind,
we work primarily under stationarity. In
applications where the chain is only approximately stationary, one can
combine burn-in, warm-start, or concentration estimates with the stationary
bounds studied here; developing this systematically is an interesting
direction beyond the scope of the present paper.

\section{Singular-value gap controls empirical variance} \label{main results section}

We will study the empirical average $\mu_t g$ through the Cesàro mean
\begin{align} \label{Cesaro mean definition}
A_t g:=\frac1t\int_0^t P_s g\,ds.
\end{align}
Note that $A_t g$ is a function on the state space $S$, whereas $\mu_t g$ is a $\mathds C$-valued random variable. They are related by 
\[
    \E_x [\mu_t g ]= A_t g \,(x)
\]
Consider a subspace $U \subset L^2_0 (\mu)$.
Analogous to $\Delta_t^U, \Delta_t$ in \eqref{definition of empirical Delta_t^U} \eqref{definition of empirical Delta_t}, we define
\begin{align*}
\delta_t^U
&:=
\sup_{g \in U: \, \|g \| = 1} \|A_t g \| 
 = \sup_{g \in U + \mathds{C}\mathbf{1}: \, \|g - \mu g\| = 1} \| A_t g - \mu g\|, \\
 \delta_t
&:=\sup_{g \in L^2_0(\mu): \, \|g\| = 1} \| A_t g \|   =
\sup_{g \in L^2 (\mu): \, \|g - \mu g\| = 1} \|A_t g - \mu g\|.
\end{align*}
We record a simple lemma for future use:
\begin{lemma}[Well-known stationary covariance identity] \label{covariance identity}
For $u\ge s$ and $g,h \in L^2(\mu)$,
\begin{align*}
\mathds{E}_\mu[\overline{g(X_s)}h (X_u)]
=
\langle g,P_{u-s}h\rangle.
\end{align*}
\end{lemma}

\begin{proof}
By the Markov property,
\begin{align*}
\mathds{E}_{\mu} \big[\overline{g(X_s)}h(X_u) \big]
=
\mathds{E}_{\mu}\!\big[
\overline{g(X_s)}\,\mathds{E}_{\mu}[h(X_u)\mid\mathcal F_s]
\big]
=
\mathds{E}_{\mu}\!\big[
\overline{g(X_s)}\,(P_{u-s}h)(X_s)
\big].
\end{align*}
Since the process is stationary, $X_s\sim \mu$, so
\begin{align*}
\mathds{E}_{\mu}[\overline{g(X_s)}h(X_u)]
=
\int \overline{g(x)}\,(P_{u-s}h)(x)\,\mu(dx)
=
\langle g,P_{u-s}h\rangle.
\end{align*}
\end{proof}

\subsection{Local singular-value gap}
\subsubsection{Upper bound in Theorem \ref{main subspace dependent theorem}}
The following proposition is inspired by Lemma 3.1 in \cite{chatterjee}.
\begin{proposition}[Cesàro-average upper bound] \label{Cesaro average upper bound}
Let $V \subset  L^2_0 (\mu)$ be a subspace such that $U = L(V \cap \Dom(L)) \subset L^2_0 (\mu)$.
Assume that the local gap defined in \eqref{gap definition subspace} satisfies $\gamma_V > 0$. Then
\begin{align*}
    \delta_t^U \leq \frac{2}{t \mkern1mu \gamma_V}.
\end{align*}
\end{proposition}
\begin{proof}
Consider a function $g \in U$. Choose $f \in V \cap \Dom(L)$ such that $-Lf = g$.
Using the semigroup identity $\int_0^t P_sLf\,ds=P_t f-f$ (see, for instance, Lemma II.1.3 (iv) in \cite{engel_nagel_2000}), we get
$\int_0^t P_s g\,ds
=
-\int_0^t P_sLf\,ds
=
f-P_t f$.
Dividing by $t$ yields
\begin{align} \label{A_tg identity in terms of f}
A_t g=\frac1t(f- P_tf).
\end{align}
Now $P_t$ is an $L^2(\mu)$-contraction, so
\begin{align*}
\|A_t g\|
\le \frac1t\bigl(\|f\|+\|P_t f\| \bigr)
\le \frac{2}{t}\|f\|.
\end{align*}
By the definition of $\gamma_V$,
$\|g\|=\|Lf\|\ge \gamma_V\|f\|$,
hence
\begin{align*}
\|f\|\le \frac{1}{\gamma_V}\|g\|.
\end{align*}
Combining the two inequalities above, we get
\begin{align*}
\|A_t g\|\le \frac{2}{t \mkern1mu \gamma_V}\|g\|.
\end{align*}
\end{proof}

With the above proposition and Lemma \ref{covariance identity}, we are ready to prove 
the upper bound on the empirical variance of observables in $U \subset L^2_0 (\mu)$.
The identities \eqref{variance of empirical average 1} \eqref{variance of empirical average 2} 
are standard in the literature on the 
ergodic average of Markov processes.
\begin{proof}[Proof of upper bound in Theorem \ref{main subspace dependent theorem}]
Let $g \in U$. With
$\mu_t g=\frac1t\int_0^t g(X_s)\,ds$,
we have
\begin{align} 
\Var_\mu (\mu_t g)
&=
\frac{1}{t^2}
\mathds{E}_{\mu}\, \Big|\int_0^t g(X_s)\,ds \Big|^2  \notag \\
&=
\frac{1}{t^2}
\int_0^t\int_0^t
\mathds{E}_{\mu}[\overline{g(X_s)}g(X_u)]\,ds\,du \notag \\
&=
\frac{2}{t^2} \Re
\int_0^t\int_0^u
\mathds{E}_{\mu}[\overline{g(X_s)}g(X_u)]\,ds\,du \notag \\
&=
\frac{2}{t^2} \Re
\int_0^t\int_0^u
\langle g,P_{u-s}g\rangle\,ds\,du \notag \\
&= \frac{2}{t^2} \Re
\int_0^t\int_0^u
\langle g,P_r g\rangle\,dr\,du \label{variance of empirical average 1} \\
&= \frac{2}{t^2} \Re \int_0^t (t-r) \langle g, P_r g \rangle \, dr. \label{variance of empirical average 2}
\end{align}
For $0\le u \leq t$, define
\[
h_u:=\int_0^u P_r g\,dr=u \,A_ug.
\]
Then
\[
\Var_\mu (\mu_t g)
=
\frac{2}{t^2}
\int_0^t \Re \langle g,h_u\rangle\,du.
\]
By Cauchy--Schwarz and Proposition \ref{Cesaro average upper bound},
\begin{align*}
\Re \langle g,h_u\rangle\ \leq
|\langle g,h_u\rangle|
\le
\|g\|\,\|h_u\|
=
u \, \|g\|\,\|A_u g\|
\le
\frac{2}{\gamma_V}\|g\|^2.
\end{align*}
Hence
\begin{align} \label{empirical average upper bound}
\Var_\mu (\mu_t g)
\le
\frac{2}{t^2}\int_0^t \frac{2}{\gamma_V}\|g\|^2\,du
=
\frac{4}{t \mkern1mu \gamma_V}\|g\|^2.
\end{align}
\end{proof}

\begin{remark}[Alternative sufficient condition]
    If we start with a subspace $U \subset L^2_0(\mu)$ with a locally bounded Poisson equation solver 
    $R_U: U \to \Dom(L)$ such that
    for all $g \in U$
    \[
    -L R_U g = g, \qquad \|R_U g\| \leq \kappa_U \| g \|,
    \]
    then the above proof and conclusion still hold with $\kappa_U^{-1}$ playing the role of $\gamma_V$.
\end{remark}

\begin{remark}[Constant improvement to empirical variance upper bound] 
\label{constant improvement upper bound remark}
    Section 3 in \cite{eberle-lorler} considers the global singular-value gap $\gamma$ in connection with \(L_0^2(\mu)\)-relaxation of nonreversible lifts.
It is also claimed there that for $g \in L^2_0(\mu)$,
\begin{align*}
    \Var_\mu (\mu_t g) \leq \frac{2}{\gamma t} \| g \|^2.
\end{align*}
We are not able to find an explicit proof for this claimed upper bound on the finite-time variance of the empirical average. This bound as written does not follow from the argument used here, but we do have a factor of 2 
improvement to \eqref{empirical average upper bound} for large time,
thanks to Theorem \ref{finite time variance bound by asymptotic variance}:
\[
        \operatorname{Var}_\mu(\mu_tg)
        \le
        \left(
        \frac{2}{\gamma t}
        +
        \frac{4}{\gamma^2t^2}
        \right)\|g\|^2 .
\]
\end{remark}

\subsubsection{Lower bound in Theorem \ref{main subspace dependent theorem}}

Similarly to the proof of the upper bound, we control $\Delta_t^{U+V}$ through $\delta_t^{U + V}$.
It is perhaps unnecessary that the lower bound involves 
observables in $V$ as well as $U$; a more refined understanding
of the lower bound would be interesting. 

\begin{proposition} \label{lower bound proposition}
    For every $t>0$,
\begin{align*}
\sup_{s \in [t,2t]} \delta_s^{U + V}
\ge
\frac{1}{\frac32   t\mkern1mu \gamma_V+3}.
\end{align*}
\end{proposition}
\begin{proof}
    Fix $t>0$, and set
\[
M_t:=\sup_{s \in [t,2t]} \delta_s^{U + V}.
\]
Let $\varepsilon>0$. By the definition of $\gamma_V$, there exists
$f\in V$
such that
\[
\|f\|=1,
\qquad
\|Lf\|\le \gamma_V +\varepsilon.
\]
Set
\[
g:=-Lf \in U.
\]
Now define
\[
u_t:=f-\frac1t\int_t^{2t}P_s f\,ds,
\qquad
v_t:=\frac1t\int_t^{2t}P_s f\,ds.
\]
Then
$f=u_t+v_t$,
so
\[
1=\|f\|\le \|u_t\|+\|v_t\|.
\]
By \eqref{A_tg identity in terms of f},
\[
u_t
=
\frac1t\int_t^{2t}(f-P_s f)\,ds
=
\frac1t\int_t^{2t} s\,A_s g\,ds.
\]
Taking norms, we get
\begin{align*}
\|u_t\|
&\le
\frac1t\int_t^{2t} s\,\|A_s g\|\,ds \\
&\le
\frac1t\int_t^{2t} s\,\delta_s^U \|g\|\,ds \\
&\le
\frac{M_t\|g\|}{t}\int_t^{2t} s\,ds \\
&\le
\frac32\,t\,M_t\,(\gamma_V+\varepsilon).
\end{align*}
Next,
\[
v_t = \frac1t\int_t^{2t}P_s f\,ds
=
2A_{2t}f-A_tf,
\]
hence
\begin{align*}
\|v_t\|
\le
2\|A_{2t}f\|+\|A_tf\|
\le
(2 \delta_{2t}^V+\delta_t^V)\|f\|
\le 3M_t.
\end{align*}
Combining the estimates,
\[
1\le \|u_t\|+\|v_t\|
\le
\left(\frac32\,t\,(\gamma_V +\varepsilon)+3\right)M_t.
\]
Letting $\varepsilon\downarrow 0$, we obtain
\[
\sup_{t\le s\le 2t}\delta_s^{U + V} = M_t \ge \frac{1}{\frac32 t \mkern1mu \gamma_V +3}.
\]
\end{proof}

\begin{proof} [Proof of lower bound in Theorem \ref{main subspace dependent theorem}]
Recall that for $h\in L^2_0(\mu)$,
\[
A_t h(x)=\mathds E_{x} [\mu_t h].
\]
By Jensen's inequality,
\begin{align} \label{Jensen}
\|A_t h\|
=
\Big(  \int \big| \E_x[\mu_t h] \big|^2 \, d\mu(x) \Big)^{1/2}
\le
(\mathds E_\mu\, |\mu_t h|^2 )^{1/2} =  \sqrt{\Var_\mu (\mu_t h)}.
\end{align}
Taking the supremum over $h \in U + V, \|h\|=1$ gives
\[
\delta_t^{U+V} \le \Delta_t^{U+V}.
\]
Therefore, by Proposition \ref{lower bound proposition},
\[
\sup_{s \in [t,2t]}\Delta_s^{U+V}
\ge
\frac{1}{\frac32 t \mkern1mu \gamma_V+3}.
\]
\end{proof}

As in \cite{chatterjee}, the dyadic lower bound also implies a uniform lower
bound at small times, showing that empirical averages
cannot converge in a time much shorter than $\gamma^{-1}$.
We shall state the global version here for simplicity. 
\begin{corollary}[A small-time lower bound]
    For $t \leq 1/(3\gamma)$, we have 
    \[
    \Delta_t \geq \frac{1}{36}
    \]
\end{corollary}
\begin{proof}
Define
\[
        F(t):=t\Delta_t .
\]
Then \(F\) is subadditive. Indeed, for \(s,t>0\), stationarity gives
\[
       (s+t) \, \mu_{s+t}g
        =
        t \cdot\frac1t\int_0^t g(X_r)\,dr
        +
        s \cdot \frac1s\int_t^{t+s} g(X_r)\,dr .
\]
Taking the \(L^2(\mathds{P}_\mu)\)-norms, using the triangular inequality, and then
taking the supremum over
\(\|g\|=1\), we obtain
\[
        (s+t)\Delta_{s+t}\le t\Delta_t+s\Delta_s.
\]
Thus, \(F(s+t)\le F(s)+F(t)\). We also have the trivial bound \(\Delta_t\le1\). Indeed, for
\(\|g\|=1\), Jensen's inequality and stationarity give
\[
        \mathds E_\mu|\mu_tg|^2 = 
        \mathds{E}_\mu \left| \frac{1}{t} \int_0^t g(X_r) \, dr \right|^2 
        \le
        \frac1t\int_0^t\mathds E_\mu|g(X_r)|^2\,dr
        =
        \|g\|^2
        =
        1.
\]

Now, suppose \(0<t\le 1/(3\gamma)\). Applying the lower bound in Corollary
\ref{main global corollary} at time scale \(12t\), 
there exists \(s\in[12t,24t]\) such that
\[
        \Delta_s
        \ge
        \frac{1}{\frac32\gamma(12t)+3}
        =
        \frac{1}{18\gamma t+3}
        \ge
        \frac19 .
\]
Write
\[
        s=nt+r,
        \qquad n=\lfloor s/t\rfloor,
        \qquad 0\le r<t.
\]
Since \(s\ge12t\), we have \(r/s\le1/12\). By subadditivity,
\[
        s\Delta_s
        =
        F(s)
        \le
        nF(t)+F(r)
        \le
        nt\Delta_t+r,
\]
where the last step uses \(\Delta_r\le1\). Dividing by \(s=nt+r\) gives
\[
        \Delta_s
        \le
        \Delta_t+\frac{r}{s}
        \le
        \Delta_t+\frac1{12}.
\]
Therefore, for \(0<t\le 1/(3\gamma)\),
\[
        \Delta_t
        \ge
        \frac19-\frac1{12}
        =
        \frac1{36}.
\]
\end{proof}

\subsection{Global singular-value gap}

If we want a universal understanding of all observables and
take $U = L^2_0(\mu)$, it is a priori unclear from the local statement 
of Theorem \ref{main subspace dependent theorem}
that we should be able to find $f \in L^2_0(\mu)$ that 
solves the Poisson equation $-Lf = g$ for every $g \in L^2_0(\mu)$. The key observation in Lemma \ref{main invertible lemma} below
is that if $L$ possesses a global singular-value gap $\gamma > 0$, then 
$L$ is in fact invertible on $L^2_0(\mu)$. In particular, we get a 
bounded Poisson solver that facilitates our proof strategies. Applying
this observation to discrete-time Markov chains, we note that 
Chatterjee's main theorem in \cite{chatterjee} can be extended 
to general state space.

\begin{proof}[Proof of Corollary \ref{main global corollary}]
    By Lemma \ref{main invertible lemma}, when $\gamma > 0$, $L$ is 
    invertible on $L^2_0(\mu)$. Then the corollary follows from Theorem 
    \ref{main subspace dependent theorem} by taking $V= U =  L^2_0(\mu)$.
\end{proof}

\subsubsection{Positive gap implies bounded Poisson solver}

We state the following lemma in a general form, so that it can be applied 
to discrete-time Markov chains and invariant
subspaces $K$. For instance, $K$ can be an eigensubspace 
when the generator $L$ has discrete spectrum. 

\begin{lemma} \label{main invertible lemma}
    Let $(S,\mu)$ be a probability space.
    Let $L$ be the generator on $L^2(\mu)$, either of a strongly continuous Markov semigroup $(P_t)_{t \geq 0}$, or of the form $L = P - I$ for a discrete-time Markov chain 
    with Markov transition kernel $P$.
    Let $K \subset L^2_0(\mu)$ be a closed subspace such that 
    $P_t(K) \subset K$ for all $t \geq 0$ in continuous time, or
    \(P(K)\subset K\) in discrete time. Let $L_K$ denote the 
    generator  restricted to $K$.
    Suppose that $L$ has a singular-value gap on $K$:
    \begin{align} \label{gap definition invariant subspace}
    \gamma_K \coloneqq \inf_{f \in \Dom(L_K)\backslash\{0\}} 
    \frac{\|L_Kf\|}{\|f\|} > 0.
    \end{align} 
    Then $L_K$ is invertible with 
    \[
    \|L_K^{-1} \|_{K \to K} \leq \frac{1}{\gamma_K}.
    \]
    
    \begin{proof}
        Since $L_K$ is the generator of a contraction semigroup, the resolvent set of $L_K$
        contains the positive half-plane. Explicitly, in continuous time,
         the resolvent $R_\lambda$ of $L_K$ at $\lambda > 0$ is given by
        \[
        R_\lambda g:= (\lambda I - L_K)^{-1} g = \int_0^\infty e^{-\lambda s}P_s g\,ds,
        \quad g \in K;
        \]
        see Theorem II.1.10 in \cite{engel_nagel_2000}.
        In discrete time, for $\lambda > 0$,
        \[
            R_\lambda:= (\lambda I -L_K)^{-1} =  ((1+\lambda)I-P_K)^{-1},
        \]
        where $P_K$ is the restriction of the Markov kernel $P$ on $K$.
        Since
        $\|\frac{1}{1+\lambda}P_K \|\le \frac{1}{1+\lambda}<1$,
        the Neumann series gives
        \[
        R_\lambda
        =
        \frac{1}{1+\lambda}\sum_{n=0}^\infty \left( \frac{P_K}{1+\lambda} \right)^n.
        \]

        For $0<\lambda<\gamma_K$ and $f\in \operatorname{Dom}(L_K)$,
        \[
        \|(\lambda I-L_K)f\|
        \ge \|L_Kf\| -\lambda\|f\|
        \ge (\gamma_K-\lambda)\|f\|.
        \]
        Therefore,
        $\|R_\lambda\|
        \le \frac{1}{\gamma_K-\lambda}$.
        For $0<\lambda,\mu<\gamma_K$, the resolvent identity yields
        \[
        R_\lambda-R_\mu=(\mu-\lambda)R_\lambda R_\mu.
        \]
        Thus, $(R_\lambda)_{\lambda\downarrow 0}$ is Cauchy in operator norm. Let
        \[
        R_K:=\lim_{\lambda\downarrow 0}R_\lambda
        \qquad\text{in } \ \| \cdot \|_{K \to K}.
        \]
        We claim that $R_K=-L_K^{-1}$. Consider $g\in K$. Since for $\lambda > 0$  
        \[
            (\lambda I - L_K) R_\lambda =I,
        \]
        we have
        \[
        -L_KR_\lambda g=g - \lambda R_\lambda g \to g
        \qquad\text{as} \ \lambda \downarrow 0.
        \]
        Since $L_K$ is closed, it follows that $R_K \mkern1mu g\in \operatorname{Dom}(L_K)$ and
        \[
        -L_KR_K \mkern1mu g=g.
        \]
        Similarly, we deduce $-R_K L_K f = f$ for any $f \in \Dom(L_K)$, hence
        \[
        R_K=-L_K^{-1}.
        \]
        Finally, since $\|L_Kf\| \geq \gamma_K \|f\|$ for $f \in \Dom(L_K)$,
        we get $\|L_K^{-1}\| \leq \gamma_K^{-1}$.
    \end{proof}
\end{lemma}

\begin{remark}
    In continuous time, the above proof could be streamlined by appealing to well-known results in semigroup theory: for example, 
    Corollary II.3.6 and Proposition IV.1.3 (iii) in \cite{engel_nagel_2000}. We choose to present explicit details to treat discrete-time Markov chains on the same footing. 
\end{remark}

\begin{remark} [Relation to the potential operator]
\label{remark G neq -L_0^-1 continuous time}
    Let $K = L^2_0(\mu)$ and denote by $L_0$ the restricted generator.
    When the Markov process is reversible, a positive spectral gap $\gamma$ implies exponential decay
    of $\|P_t g\|$ for $g \in L^2_0(\mu)$, and one may write $-L_0^{-1}$ as the potential operator $G$ defined by 
\begin{align} \label{G potential operator definition}
        Gg:=\int_0^\infty P_rg\,dr.
\end{align}
 In the general nonreversible case, it is possible for $\gamma$ to be positive, therefore $-L_0^{-1}$ exists, while $Gg$ is not well-defined for some $g \in L^2_0(\mu)$. See Example 
\ref{pure rotation 1} below.
 Similarly, a positive $\gamma$ for a discrete-time chain does not guarantee
    that $P$ is strictly contractive on $L^2_0(\mu)$, so one cannot always hope to write $-L_0^{-1}$ as the potential operator $G$ defined by 
    \begin{align*}
        Gg \coloneqq \sum_{k=0}^\infty P^k g.
    \end{align*}

    In both continuous time and discrete time, if $Gg$ converges strongly for all $g \in L^2_0(\mu)$,
    then $L_0$ is invertible, and $Gg = -L_0^{-1}g$. In continuous time, this can be ensured if
    $\|P_t\|_{L^2_0 \to L^2_0} \leq Ce^{-at}$ for some $C,a > 0$. In particular,
    $P_t$ becomes exponentially contractive if we
    perturb the generator so that its symmetric part has a positive spectral gap 
    $\gamma_s > 0$; see \eqref{L2 contraction by gamma_s} below.
    
     In discrete time, the corresponding sufficient condition is $\|P\|_{L^2_0 \to L^2_0} < 1$.
     For example, if a Markov chain with transition operator $P$ has a positive singular-value gap \(\gamma>0\), then adding laziness makes the transition operator strictly contractive on
\(L^2_0(\mu)\). Indeed, for
\(\tilde P=\alpha I+(1-\alpha)P\) with $\alpha \in (0,1)$, one may check that
\[
        \|\tilde P\|_{L^2_0\to L^2_0}
        \le
        \sqrt{1-\alpha(1-\alpha)\gamma^2}<1,
\]
so that $\sum_{k=0}^\infty \tilde{P}^k$ converges in operator norm to the inverse of $I-\tilde{P}$.
\end{remark}

\begin{example}[Pure rotation on the circle] \label{pure rotation 1}
    As a simple example, consider the unit circle 
$\mathds{T}$ with normalized Lebesgue measure $\mu$, where the dynamics is given by deterministic rotation: $P_t g(\theta) \coloneqq g(\theta +  t)$, so that 
$L = \partial_\theta$. The space $L^2_0(\mu)$ is spanned by the orthonormal Fourier modes $e_k(\theta) = e^{ik\theta}$ for $k \neq 0$. Since $Le_k = i  k \mkern1mu e_k$, $L_0$ has singular-value gap $\gamma = 1$ achieved by $g = e_1$, so $L_0$ is 
invertible and $-L_0^{-1} g = i \mkern1mu g$. On the other hand, the integral defining $Gg$ does not converge:
\begin{align*}
    \int_0^t P_s g \, ds = \int_0^t e^{is} \,ds \ g = \frac{e^{it} - 1}{i} g.
\end{align*}
\end{example}

\subsubsection{Variance bounds for Markov chains on general state space}
Here we briefly record that Lemma \ref{main invertible lemma} also admits the 
extension of Theorem 1.2 in \cite{chatterjee} to general state space.
Let us first clarify the setting. Let $(X_n)_{n=0}^\infty$ be a time-homogeneous discrete time Markov chain on 
a measurable space $S$. The Markov 
transition kernel $P(x,dy)$ acts on $g \in L^2(\mu)$ by 
$Pg\,(x) = \int g(y) \,P(x,dy)$. Assume that the chain is stationary with $X_0 \sim \mu$, where 
$\mu$ is an invariant measure satisfying $\int P(x,dy) \, \mu(dx) = \mu(dy)$.
Use the same notation as above and as in \cite{chatterjee}:
\[
    \mu g:=\int g\,d\mu,
\qquad \qquad
\mu_n g:=\frac1n \sum_{i=0}^{n-1} g(X_i) 
\]
As above, use $\|g\|$ to denote the $L^2(\mu)$ norm of a function $g$ on the state space. The 
worst-case standard deviation of empirical averages at time $n$ is captured by
\[
\Delta_n
:= \sup_{g \in L^2_0(\mu):\, \|g\| = 1} \sqrt{\Var_\mu (\mu_n g)}
= \sup_{g \in L^2 (\mu): \, \|g - \mu g\| = 1} \sqrt{\Var_\mu(\mu_n g)}
\]
In this case, the generator is $L = P-I$, and we consider the singular-value gap 
defined in \eqref{gap definition discrete time}.
Note that since $P$ is a contraction on $L^2(\mu)$, $L$ is a bounded operator and we do 
not need to worry about its domain.

We have the exact analog of Theorem 1.2 in \cite{chatterjee} in this setting.
A similar local statement holds with the same proof, but we only state the global 
theorem for simplicity. 
\begin{theorem} Assume $\gamma > 0$.
    For $n \geq 1$, we have the following upper and lower bounds on $\Delta_n$:
    \begin{align*}
        &\text{(i)}  \qquad  \Delta_n \leq \frac{2}{\sqrt{\gamma n}}\\
        &\text{(ii)} \max_{n \leq k \leq 2n} \Delta_k \geq \frac{1}{2\gamma n +3}.
    \end{align*}
\end{theorem}
\begin{proof}
    The exact proof in \cite{chatterjee} goes through, with the following two modifications. For the upper bound,
    finite state space only enters the original proof in \cite{chatterjee} through a dimension-counting argument 
    to ensure that $(P-I)$ is invertible on $L^2_0(\mu)$ when $\gamma > 0$. This argument can be replaced 
    by Lemma \ref{main invertible lemma}. 

    For the lower bound, finite state space only enters the original proof with the choice of 
    a minimizer $f$ of $\|(P-I)f\|$ with $\|f\| = 1$. This can be resolved by considering near-minimizers, as we did in the proof of Proposition \ref{lower bound proposition}.
\end{proof}

\subsection{Sharpness of upper and lower bounds} \label{sharpness of bounds subsection}
Although they show that $t \gtrsim \gamma^{-1}$ is sufficient and necessary for empirical averages 
to converge,
the upper and lower bounds in Corollary \ref{main global corollary} should not be interpreted as a
two-sided equivalence for \(\Delta_t\). Rather, they capture two different
universal obstructions. The upper bound
\[
        \Delta_t \le \frac{2}{\sqrt{\gamma \mkern1mu t}}
\]
is sharp in general. In particular, it has the correct order $t^{-1/2}$ when 
the observable in question has a finite positive asymptotic variance.
On the other hand, the lower bound
\[
        \sup_{s\in[t,2t]}\Delta_s
        \ge
        \frac{1}{\frac32 \gamma  \mkern1mu t +3}
\]
is also sharp as a universal lower bound: without additional assumptions,
one cannot replace it by a lower bound of order \(t^{-1/2}\).
We first illustrate the sharpness of the upper bound with
a familiar example.

\begin{example}[Stationary iid sampling]
Let \((S,\mu)\) be a probability space, and let \(X_t\) be the following
stationary Markov process. At exponential rate $q$, 
the current state is replaced by an independent sample from \(\mu\). The generator of 
this process is
\[
        Lf(x)=q\int (f(y)-f(x))\,\mu(dy).
\]
Thus, on \(L_0^2(\mu)\),
\[
        Lf=-qf,
        \qquad
        P_t f=e^{-qt}f .
\]
Hence the global singular-value gap is
\[
        \gamma=q.
\]
For any \(g\in L_0^2(\mu)\), we have
$\langle g, P_r g \rangle = 
        e^{-qr}\|g\|^2$.
Therefore, using \eqref{variance of empirical average 1},
\begin{align*}
        \operatorname{Var}_\mu(\mu_t g)
        &=
        \frac{2}{t^2}\int_0^t \int_0^u e^{-qr}\,dr\, du\,\|g\|^2  \\
        &=
        \left(
        \frac{2}{qt}
        -
        \frac{2(1-e^{-qt})}{q^2t^2}
        \right)\|g\|^2 .
\end{align*}
Consequently, as \(t\to\infty\),
\[
        \Delta_t
        \sim
        \sqrt{\frac{2}{\gamma t}} .
\]
Thus, the upper bound in Corollary \ref{main global corollary} has the correct order in both \(t\)
and \(\gamma\), up to constants.
The discrete-time analog is the independent-sampling chain
\(P(x,\cdot)=\mu(\cdot)\). Then \(P-I=-I\) on \(L_0^2(\mu)\), so
\(\gamma=1\), and empirical averages have standard deviation exactly of order \(n^{-1/2}\).
\end{example}

The following example shows that the lower bound in Corollary \ref{main global corollary}
is also sharp. Note that in this case, the lower
bound in Theorem \ref{finite time variance bound by asymptotic variance} below 
does not imply a $t^{-1/2}$ order for the empirical standard deviation, because the 
asymptotic variance vanishes.
\begin{example}[Pure rotation on the circle, revisited] \label{pure rotation on circle with alpha drift}
Let \(S=\mathds T=\mathds R/(2\pi\mathds Z)\), \(\mu\) be the normalized
Lebesgue measure, and consider deterministic rotation
\[
        X_t=X_0+\alpha t \mod{2\pi},
        \qquad \alpha\ne0 .
\]
The semigroup and the generator are given by
\[
        P_tg(\theta)=g(\theta+\alpha t), \qquad
        L=\alpha\partial_\theta.
\]
For the Fourier modes \(e_k(\theta)=e^{ik\theta}\), \(k\in\mathds Z\),
 $Le_k=i\alpha k e_k$, so
\[
        \gamma=\inf_{k\ne0}|i\alpha k|=|\alpha|.
\]
Since the dynamics is deterministic and $\mu$ is invariant, for centered
\(g\) we have $\mu_t g = A_t g (X_0)$. As $X_0 \sim \mu$, we get
\begin{align} \label{finite time variance equals Cesaro average norm squared}
        \operatorname{Var}_\mu(\mu_tg) = \E_\mu |\mu_t g|^2    = \int |A_t g(x) |^2 \, \mu(dx)      =\|A_tg\|^2 .
\end{align}
For a Fourier mode,
\[
        A_te_k = \frac{1}{t} \int_0^t e^{i\alpha k s} e_k \, ds
        =
        \frac{e^{i\alpha kt}-1}{i\alpha kt}e_k,
\]
hence
\[
        \Delta_t
        =
        \sup_{k\ne0}
        \left|
        \frac{e^{i\alpha kt}-1}{i\alpha kt}
        \right|
        \le
        \frac{2}{|\alpha|t}
        =
        \frac{2}{\gamma t}.
\]
Thus, empirical standard deviation can decay at order \((\gamma t)^{-1}\). In particular, no universal lower bound of
order \((\gamma t)^{-1/2}\) can hold under only the assumption of a
positive singular-value gap.

Moreover, we can explicitly check that the lower bound of
the dyadic supremum in Corollary \ref{main global corollary} has the correct order. If
\(|\alpha|t\ge 2\pi\), then the interval \([t,2t]\) contains some
\(s\) for which \(e^{i\alpha s}=-1\). For this \(s\),
\[
        \Delta_s
        \ge
        \left|
        \frac{e^{i\alpha s}-1}{i\alpha s}
        \right|
        =
        \frac{2}{|\alpha|s}
        \ge
        \frac{1}{|\alpha|t}
        =
        \frac{1}{\gamma t}.
\]
Thus, for large \(\gamma t\),
\[
        \sup_{s\in[t,2t]}\Delta_s
        \asymp
        \frac{1}{\gamma t}.
\]
\end{example}

\subsection{Sharpened bounds with access to asymptotic variance}
The preceding examples show that the singular-value gap alone cannot give
a two-sided characterization of \(\Delta_t\) at the order
\(t^{-1/2}\). A sharper description is possible when 
we can locate a set of observables for which the
asymptotic variance is computable or can be bounded. 
In particular, when we have a local gap $\gamma_K > 0$ on a closed invariant 
subspace $K$,
the asymptotic variance can be uniformly bounded above by $\gamma_K^{-1}$.
As observed in Remark \ref{constant improvement upper bound remark}, this 
provides a constant improvement to the 
upper bound stated in Theorem \ref{main subspace dependent theorem}.
More importantly, the lower bound in \eqref{bound using asymptotic variance and gap, observable specific}
improves the $t$-dependence of the lower bound in Theorem \ref{main subspace dependent theorem}
when $\sigma^2_g > 0$ for an observable $g$ in question, 
bounding the finite-time standard deviation of empirical averages 
within the CLT scale of $t^{-1/2}$.
This leads to an improved observable-class lower bound whenever the
worst-case asymptotic variance constant \(\sigma_K^2\) is positive.

Let \(K\subset L_0^2(\mu)\) be a closed invariant subspace to which
Lemma \ref{main invertible lemma} applies.
Assume \(\gamma_K>0\), so that \(L_K^{-1}:K\to K\) exists and
$\|L_K^{-1}\|\le \gamma_K^{-1}$.
For \(g\in K\), consider its asymptotic variance $\sigma_g^2$
defined in \eqref{asymptotic variance definition}. 
By Lemma \ref{appendix lemma proof of asymptotic variance formula},
we have
\[
        \sigma_g^2
        =
        2\operatorname{Re}\langle g,-L_K^{-1}g\rangle .
\]
We also define the worst-case asymptotic variance
constant on \(K\) by
\begin{align} \label{definition of worst case asymptotic variance}
        \sigma_K^2
        :=
        \sup_{\substack{g\in K\\ \|g\|=1}}
        \sigma_g^2
        =
        2
        \sup_{\substack{g\in K\\ \|g\|=1}}
        \operatorname{Re}\langle g,-L_K^{-1}g\rangle .
\end{align}
Since \(\|L_K^{-1}\|\le \gamma_K^{-1}\), we have
\begin{align} \label{uniform upper bound of asymptotic variance by 2/gamma}
        0\le \sigma_K^2\le \frac{2}{\gamma_K}.
\end{align}
\begin{theorem}[Finite-time comparison with asymptotic variance]
\label{finite time variance bound by asymptotic variance}
In the above setting, 
\begin{align} \label{bound using asymptotic variance and gap, observable specific}
        \left|
        \operatorname{Var}_\mu(\mu_tg)
        -
        \frac{\sigma_g^2}{t}
        \right|
        \le
        \frac{4}{\gamma_K^2t^2}\|g\|^2 
\end{align}
for every \(g\in K\).
Consequently,
\begin{align} \label{bound using both asymptotic variance and gap}
        \frac{\sigma_K^2}{t}
        -
        \frac{4}{\gamma_K^2t^2}
        \le
        \bigl(\Delta_t^K\bigr)^2
        \le
        \frac{\sigma_K^2}{t}
        +
        \frac{4}{\gamma_K^2t^2}.
\end{align}
\end{theorem}

\begin{proof}
We again use the identity \eqref{variance of empirical average 1}:
\[
        \operatorname{Var}_\mu(\mu_tg)
        =
        \frac{2}{t^2}
        \operatorname{Re}
        \int_0^t \int_0^u \langle g,P_rg\rangle\,dr \,du .
\]
Recall that $\int_0^u P_r g \,ds = f - P_uf$ for $f = -L_K^{-1} g$. Since $-L_K^{-1}$ 
commutes with the semigroup,
\begin{align*}
        \int_0^t \int_0^u P_r g \, dr \, du &= \int_0^t  -L_K^{-1} (I - P_u)g \,du
        \\
        &=
        -tL_K^{-1}g
        +
        L_K^{-2}(P_t - I)g .
\end{align*}
Hence
\begin{align*}
        \operatorname{Var}_\mu(\mu_tg)
        &=
        \frac{2}{t}
        \operatorname{Re}\langle g,-L_K^{-1}g\rangle
        +
        \frac{2}{t^2}
        \operatorname{Re}\langle g,L_K^{-2}(P_t- I)g\rangle  \\
        &=
        \frac{\sigma_g^2}{t}
        +
        \frac{2}{t^2}
        \operatorname{Re}\langle g,L_K^{-2}(P_t- I)g\rangle .
\end{align*}
Using \(\|L_K^{-1}\|\le \gamma_K^{-1}\) and \(\|P_t\|\le1\),
\[
        \left| \operatorname{Re}\langle g,L_K^{-2}(P_t- I)g\rangle  \right| \leq
        \left|
        \langle g,L_K^{-2}(I-P_t)g\rangle
        \right|
        \le
        \|g\|\,\|L_K^{-2}\|\,\|(I-P_t)g\|
        \le
        \frac{2}{\gamma_K^2}\|g\|^2 .
\]
This proves the first inequality. Taking the supremum over functions $g \in K$ with
\(\|g\|=1\) gives the second pair of inequalities.
\end{proof}

We note that when the process is nonreversible, the variance of finite-time empirical averages 
might not always be trivially controlled by the asymptotic variance.
For one thing, the lower bound in $\eqref{bound using asymptotic variance and gap, observable specific}$ is vacuous when the asymptotic variance vanishes. We saw that 
the empirical standard deviation can indeed decay faster at the order of $t^{-1}$ in 
Example \ref{pure rotation on circle with alpha drift}.
Moreover, the upper bound in
\eqref{bound using both asymptotic variance and gap} can fail
without a positive singular-value gap:
see Example \ref{2d irrational rotation} below. 
This is in contrast to the 
reversible scenario, where the finite-time variance 
is bounded by the asymptotic variance whenever the latter is finite. Since
this bound has no direct reference to the spectral gap,
we can get a uniform upper bound for finite-time variance of a class of observables if we
can bound their asymptotic variance uniformly.
The aforementioned bound can be quickly deduced from \eqref{variance of empirical average 2} and 
the observation that $P_{r/2} = P_{r/2}^*$ for $r > 0$ when the process is reversible (cf.
Proposition 4.29 in 
\cite{aldous-fill:book} and Lemma 2.3 in \cite{cattiaux2012central}):
\begin{align}
    \Var_\mu (\mu_t g) &= \frac{2}{t^2} \int_0^t (t-r) \Re \langle g, P_r g \rangle  \, dr \nonumber \\
    &= \frac{4}{t^2} \int_0^{t/2} (t-2s) \| P_s g \|^2 \, ds \label{variance label one} \\
    &\leq \frac{4}{t} \int_0^\infty \|P_s g\|^2 \, ds  \nonumber \\
    &= \frac{\sigma_g^2}{t}, \label{variance label two}
\end{align}
where the last equality in \eqref{variance label two} follows from the expression of asymptotic variance obtained by
multiplying \eqref{variance label one} by
$t$ and taking the limit $t \to \infty$.

\begin{example}[Pure rotation on 2d torus] \label{2d irrational rotation}
    Consider pure rotation again, but this time on the $2d$ torus $\mathds T^2$ with drift $\alpha = (1, \beta)$ for irrational $\beta$.
    The Fourier modes $e_k(\theta) = e^{ik \cdot \theta}$ indexed by $k \in \mathds{Z}^2$ have eigenvalues 
    $ik \cdot \alpha$. Write $\xi_k \coloneqq k \cdot \alpha$, the frequency of the $k$-th mode. 
    Since $|\xi_k|$ can be arbitrarily small, the process does not 
    have a gap on $L^2_0(\mathds{T}^2)$. 

    Since the dynamics is deterministic and the normalized Lebesgue measure $\mu$ is invariant, we 
    have $\Var_\mu(\mu_t g) = \|A_t g\|^2$ as in \eqref{finite time variance equals Cesaro average norm squared}.
    Using $P_s  e_k(\theta)= e^{i \xi_k s} e_k $, 
    \[
    A_t e_k = \frac{1}{t} \int_0^t P_s e_k \, ds = \frac{e^{i\xi_kt}-1}{i \xi_k t} e_k.
    \]
    We therefore get 
    \[
    \Var_\mu (\mu_t e_k) = \left| \frac{e^{i\xi_kt}-1}{i \xi_k t} \right|^2 = \frac{4\sin^2(\xi_k t/2)}{\xi_k^2 t^2}
    \implies \sigma^2_{e_k} = \lim_{t \to \infty} t \Var_\mu (\mu_t e_k) = 0
    \]
    However, for each $t > 0$, we can choose
    $k$ such that $t|\xi_k|  \ll 1$, so that 
    \[
      \Var_\mu(\mu_t e_k) =        \left\|  \frac{1}{t} \int_0^t e^{i \xi_k s} e_k \, ds  \right\|^2 \approx 1.
    \]
    Therefore, $\Delta_t = 1$ for every $t$. This example shows that asymptotic variance alone cannot control finite-time empirical variance uniformly over an infinite-dimensional observable class.
    Here, every fixed Fourier mode has zero asymptotic variance, but the
    frequencies accumulate at zero, so for each fixed \(t\) we can choose a
    mode that has barely moved. In particular, the infinite-time limit does not commute 
    with the worst-case supremum over observables:
    \[
    \sup_{k \neq 0} \lim_{t \to \infty} \, t \Var_\mu (\mu_t e_k) = 0, \quad \text{while} \quad 
    \limsup_{t \to \infty} \sup_{k \neq 0} \Var_\mu (\mu_t e_k) = 1.
    \]
\end{example}

\newpage
\section{Ornstein--Uhlenbeck process with rotation} \label{OU process section}
In the following example, the generator is normal despite not being self-adjoint in $L^2(\mu)$.
Moreover, the spectrum and eigenfunctions are explicit, so we get exact global and 
local singular-value gaps on eigensubspaces. In this case, the variance of finite-time empirical averages
is computable in the eigenbasis, so our theory in section \ref{main results section} above 
strictly speaking does not provide more information. Nevertheless, the theory helps
us obtain clean bounds for natural classes of observables, and
this example serves as a good illustration
for various aspects of the bounds we derived. Let
\[
        J=
        \begin{pmatrix}
        0&-1\\
        1&0
        \end{pmatrix},
\]
and consider the diffusion on \(\mathds R^2\) with a rotation of strength $\alpha$:
\[
        dX_t=(-X_t+\alpha JX_t)\,dt+\sqrt{2}\,dW_t,
        \qquad \alpha\in\mathds R.
\]
The invariant measure is the standard Gaussian measure
\[
        \mu(dx)=\frac{1}{2\pi}e^{-|x|^2/2}\,dx = \frac{1}{2\pi} e^{-r^2/2} \,r \, dr d\theta.
\]
The generator can be decomposed into symmetric and antisymmetric parts:
\[
        L_\alpha
        =
        \Delta-x\cdot\nabla+\alpha\partial_\theta  = S + \alpha A,
\]
where the symmetric part is the usual reversible Ornstein--Uhlenbeck generator
$S=\Delta-x\cdot\nabla$,
and the antisymmetric part is \(\alpha A = \alpha \partial_\theta\). Since \(S\) is
rotationally invariant, \(S\) commutes with $A$. $L_\alpha$ is then normal, and
in fact its eigenfunctions can be written explicitly in
terms of generalized Laguerre polynomials. See Appendix \ref{Laguerre appendix} for details.
We have \((\psi_{k,m})_{k\ge0,m\in\mathds Z}\) an orthonormal basis of
\(L^2(\mu)\), such that
\begin{align}\label{action of symmetric and antisymmetric parts on psi_k,m}
        S\psi_{k,m}=-(2k+|m|)\psi_{k,m},
        \qquad
        \partial_\theta\psi_{k,m}=im\psi_{k,m},
\end{align} 
so that
\[
        L_\alpha\psi_{k,m}
        =
        \bigl(-(2k+|m|)+i\alpha m\bigr)\psi_{k,m}.
\]
This decomposition is a polar-coordinate version of the Hermite chaos
decomposition for the Gaussian measure. See \cite{ChenLiu_2014} for a related
treatment of complex OU operators. 

\subsection{Explicit variance computation}
Since the eigenbasis is explicit, one can also compute the empirical
variance mode by mode. Let
\[
        n=2k+|m|,
        \qquad
        \omega=\alpha m.
\]
Then
\[
        L_\alpha\psi_{k,m}=(-n+i\omega)\psi_{k,m}.
\]
For a normalized eigenfunction \(\psi_{k,m}\), the identity \eqref{variance of empirical average 1} gives
\begin{align} \label{2d OU empirical variance}
        \operatorname{Var}_\mu(\mu_t\psi_{k,m})
        =
        \frac{2}{t^2}
        \operatorname{Re}
        \left[
        \frac{e^{(-n+i\omega)t}-1-(-n+i\omega)t}
        {(-n+i\omega)^2}
        \right].
\end{align}
In particular,
\begin{align} \label{2d OU asymptotic variance}
        \sigma^2_{\psi_{k,m}} = 
        \lim_{t\to\infty}
        t\,\operatorname{Var}_\mu(\mu_t\psi_{k,m})
        =
        \frac{2n}{n^2+\omega^2}
        =
        \frac{2(2k+|m|)}
        {(2k+|m|)^2+\alpha^2m^2}.
\end{align}
For any observable $g \in L^2_0(\mu)$, if we can write it in the eigenbasis as 
$g = \sum_{k,m} c_{k,m} \psi_{k,m}$, we get the empirical variance explicitly
\[
\Var_\mu (\mu_t g) = \sum_{k,m} |c_{k,m}|^2 \Var_\mu (\mu_t \psi_{k,m})
\]

\subsection{Uniform acceleration for nonradial observables}
Since \(L_\alpha\) is normal, the singular values of \(L_\alpha\) are the
moduli of its eigenvalues. On \(L_0^2(\mu)\), the constant mode
\((k,m)=(0,0)\) is excluded, so the global singular-value gap is
\[
\begin{aligned}
        \gamma
        =
        \inf_{(k,m)\ne(0,0)}
        \left|
        -(2k+|m|)+i\alpha m
        \right| 
         =
        \min\{\sqrt{1+\alpha^2},2\},
\end{aligned}
\]
where the modes \((k,m)=(0,\pm1)\) give \(\sqrt{1+\alpha^2}\), while the
radial mode \((k,m)=(1,0)\) gives \(2\). 
Consequently, the global singular-value gap 
is capped at 2 when the antisymmetric rotation strength
$\alpha$ grows. The obstruction is the radial quadratic observable
\[
       \psi_{1,0} \propto |x|^2-2,
\]
whose empirical average does not converge faster when a rotation is present.

This obstruction shows that the
subspace formulation of Theorem \ref{main subspace dependent theorem}
could be useful when the irreversible part of the dynamics accelerates the
convergence of empirical averages in certain 
directions but leaves some other observables invariant, in which case
we may wish to separately consider different observable subspaces. 
In this case, to avoid the obstruction presented by $\psi_{1,0}$, consider
\[
        K_{\mathrm{ang}}
        :=
        \overline{
        \operatorname{span}\{\psi_{k,m}:m\ne0\}
        } = \left\{f \in L^2_0(\mu): \int_0^{2\pi} f(r,\theta) \, d\theta = 0 
        \ \text{ for a.e. }r
        \right\},
\]
where the second characterization above can be obtained by considering the 
Fourier decomposition $f(r,\theta) = \sum_{m \in \mathds{Z}} f_m(r) e^{i m \theta}$
and noting that $\psi_{k,m}$ for $k \geq 0$ form a basis for the $m$-th mode.
Equivalently, \(K_{\mathrm{ang}}\) is the orthogonal complement of radial functions.
\[
\begin{aligned}
        \gamma_{\mathrm{ang}}
        =
        \inf_{k\ge0,\;m\ne0}
        \sqrt{(2k+|m|)^2+\alpha^2m^2}  
        =
        \sqrt{1+\alpha^2}.
\end{aligned}
\]
On this closed invariant subspace, Theorem \ref{main subspace dependent theorem}
and Lemma \ref{main invertible lemma} give
\begin{align} \label{K_ang empirical variance bound}
        \operatorname{Var}_\mu(\mu_tg)
        \le
        \frac{4}{t\sqrt{1+\alpha^2}}\|g\|^2
\end{align}
for every \(g\in K_{\mathrm{ang}}\).
Thus, rotation gives a uniform empirical-average acceleration for all
nonradial observables, clarifying, at the level of a whole observable class,
the intuition that \cite{DLP2016} verifies explicitly
for linear and quadratic observables.

\subsection{Gaussian chaos and sharpness of variance bounds}
Recall that for the standard Gaussian measure on \(\mathds R^2\), the
\(n\)-th Gaussian chaos \(\mathcal H_n\) is the eigenspace of the
reversible Ornstein--Uhlenbeck generator
        $S=\Delta-x\cdot\nabla$
with eigenvalue \(-n\). Equivalently, \(\mathcal H_n\) is spanned by
products of one-dimensional Hermite polynomials of total degree \(n\).
From \eqref{action of symmetric and antisymmetric parts on psi_k,m} we see that
\[
        \mathcal H_n
        =
        \operatorname{span}\{\psi_{k,m}:2k+|m|=n\}.
\]
In particular, for a given $n$, the possible angular frequencies
       $ m=-n,-n+2,\ldots,n-2,n$
give \(n+1\) modes, which equals \(\dim \mathcal H_n\).
We spell out the first few eigenspaces, which can be obtained from the explicit form 
of $\psi_{k,m}$ derived in Appendix \ref{Laguerre appendix}:
\begin{align*}
        \mathcal H_1
        &=
        \operatorname{span}\{\psi_{0,1},\psi_{0,-1}\}
        =
        \operatorname{span}\{x_1,x_2\}, \\
        \mathcal H_2
        &=
        \operatorname{span}\{\psi_{1,0},\psi_{0,2},\psi_{0,-2}\} =
        \operatorname{span}\{
        |x|^2-2,\,
        x_1^2-x_2^2,\,
        2x_1x_2
        \}.
\end{align*}
In words, $\mathcal{H}_1$ is the space of centered linear observables and $\mathcal{H}_2$ is 
the space of centered quadratic observables.
We compute the gap of $L_\alpha$ on the $n$-th Gaussian chaos:
\[
        \gamma_{\mathcal H_n}
        =
        \min_{m=-n,-n+2,\ldots,n}
        \sqrt{n^2+\alpha^2m^2}
         =
        \begin{cases}
        n & n \text{ even}\\[4pt]
        \sqrt{n^2+\alpha^2} & n \text{ odd}.
        \end{cases}
\]
For even \(n\), the obstruction is again the radial sector \(m=0\).
If we restrict attention to the angular part of \(\mathcal H_n\), 
the local gap increases with $\alpha$ as
\[
        \gamma_{\mathcal H_n\cap K_{\mathrm{ang}}}
        =
        \begin{cases}
            \sqrt{n^2+4\alpha^2} &n  \text{ even nonzero}\\[4pt]
        \sqrt{n^2+\alpha^2} & n \text{ odd}.
        \end{cases}
\]
Similarly, we may use \eqref{2d OU asymptotic variance} to compute
the worst-case asymptotic variance defined in \eqref{definition of worst case asymptotic variance}:
\begin{align*}
\sigma^2_{\mathcal{H}_n} = \max_{m=-n,-n+2, \ldots, n} \, \frac{2n}{n^2 + \alpha^2 m^2} &= 
\begin{cases}
     \frac{2}{n} & n  \text{ even}\\[4pt]
      \frac{2n}{n^2 + \alpha^2} & n \text{ odd}
\end{cases} \\
\sigma^2_{\mathcal{H}_n \cap K_{\text{ang}}} &= 
\begin{cases}
     \frac{2n}{n^2 +4\alpha^2} & n  \text{ even}\\[4pt]
      \frac{2n}{n^2 + \alpha^2} & n \text{ odd}
\end{cases}
\end{align*}
In particular, for fixed $n$, the asymptotic variance of observables in 
$\mathcal{H}_n \cap K_{\text{ang}}$ is of order $\alpha^{-2}$.
The upper bound in Theorem \ref{main subspace dependent theorem} gives 
\[
\operatorname{Var}_\mu(\mu_tg)
        \le
        \frac{4}{t\sqrt{n^2+\alpha^2}}\|g\|^2
\]
for all $g \in \mathcal{H}_n \cap K_{\text{ang}}$, which gives a weaker
$|\alpha|^{-1}$ dependence as $|\alpha| \to \infty$. On the other hand, Theorem
\ref{finite time variance bound by asymptotic variance} gives the upper bound
\[
(\Delta_t^{\mathcal{H}_n \cap K_{\text{ang}}})^2 \leq 
\frac{\sigma^2_{\mathcal{H}_n \cap K_{\text{ang}}}}{t} + \frac{4}{\gamma^2_{\mathcal{H}_n \cap K_{\text{ang}}}t^2}
\leq \frac{2n}{(n^2 + \alpha^2)t} + \frac{4}{(n^2 + \alpha^2) t^2},
\]
which matches the $\alpha^{-2}$ factor of variance reduction as $|\alpha| \to \infty$.

On the larger nonradial subspace \(K_{\mathrm{ang}}\), the
upper bound in Theorem \ref{main subspace dependent theorem} already sharply 
captures the $\alpha$-dependence of variance reduction. Indeed, 
using \eqref{2d OU asymptotic variance} as $|\alpha| \to \infty$,
\[
        \sigma^2 _{K_{\mathrm{ang}}}
        =
        \sup_{\substack{k\ge0\\m\ne0}}
        \frac{2(2k+|m|)}
        {(2k+|m|)^2+\alpha^2m^2}
        \asymp
        \frac{1}{|\alpha|}
        \asymp
        \frac{1}{\gamma_{\mathrm{ang}}}.
\]
The upper bound above follows from
\[
        \frac{2n}{n^2+\alpha^2m^2}
        \le
        \frac{2n}{n^2+\alpha^2}
        \le
        \frac{1}{|\alpha|}
        \qquad (m\ne0),
\]
and the matching lower bound is obtained by
taking \(|m|=1\) and choosing \(n=2k+1\) of order \(|\alpha|\).
Therefore, for the infinite-dimensional observable class
$K_{\text{ang}}$, the
\(\alpha\)-dependence in the finite-time variance bound of
Theorem \ref{main subspace dependent theorem} matches the 
reduction factor for asymptotic variance.

\subsection{Practical implication for variance reduction}
    In higher dimensions, we have more freedom to choose the perturbation $J$
    while maintaining the invariance of the Gaussian measure. It has 
    been widely observed in the context of Markov chains and diffusion processes
    that the addition of an irreversible perturbation accelerates
    convergence of the law of $X_t$ and, respectively, convergence of empirical averages.
    It is therefore of interest to find an optimal perturbation to accelerate convergence.
    
    The observable-class perspective we take here suggests an optimization problem
different from the one emphasized in \cite{DLP2016}. There, the perturbation is chosen
to minimize the asymptotic variance of a fixed observable, and their Section 4.2
solves this problem for linear observables under a Frobenius norm constraint
on \(J\). In the present framework, one can instead fix a closed invariant class of
observables \(K\), such as \(\mathcal H_1\), \(\mathcal H_n\), or a finite
sum of chaoses, and choose \(J\) to maximize the local singular-value gap
\(\gamma_K(J)\), or, often equivalently, use Theorem \ref{finite time variance bound by asymptotic variance} to minimize the worst-case
asymptotic variance \(\sigma^2_K(J)\). Solving this optimization problem gives 
a uniform convergence guarantee for a class of observables and
extends naturally beyond linear and quadratic functions.

\begin{example}[Uniform variance reduction for linear observables]
We illustrate the observable-class variance reduction perspective 
in a simple higher-dimensional Gaussian example. 
Let \(\mu=\mathcal{N}(0,I_d)\), and let \(J\in\mathds R^{d\times d}\)
be skew-symmetric. Consider
\[
        L_{J,\alpha}
        =
        \Delta-x\cdot\nabla+\alpha \, Jx \cdot\nabla ,
        \qquad \alpha\in\mathds R.
\]
which leaves \(\mathcal{N}(0,I_d)\) invariant. Let
\[
        \mathcal H_1
        =
        \{f_\ell(x)=\ell\cdot x:\ell\in\mathds R^d\}
\]
be the first Gaussian chaos, i.e. the space of centered linear observables.
Note that
\[
        L_{J,\alpha}f_\ell
        =
        -f_{(I+\alpha J)\ell}.
\]
Indeed,
\[
        \Delta f_\ell=0,\qquad
        -x\cdot\nabla f_\ell=-\ell\cdot x, \qquad
        (Jx)\cdot\nabla f_\ell
        =
        \ell\cdot Jx
        =
        (J^T\ell)\cdot x
        =
        -(J\ell)\cdot x.
\]
Since $\| f_\ell\| = \|\ell\|_{\mathds{R}^d}$, the singular-value 
gap of \(L_{J,\alpha}\) on \(\mathcal H_1\) is given by the smallest singular-value of $I + \alpha J$.
Since \(J^T=-J\),
$(I+\alpha J)^T(I+\alpha J)
        =
        I+\alpha^2J^TJ$. Hence
\[
        \gamma_{\mathcal H_1}(J)
        = \sqrt{
        1+\alpha^2\lambda_{\min}(J^TJ)}
\]
The worst-case asymptotic variance on $\mathcal{H}_1$
can also be computed explicitly. Since
\[
        -L_{J,\alpha}^{-1}f_\ell
        =
        f_{(I+\alpha J)^{-1}\ell},
\]
we get
\[
        \sigma_{f_\ell}^2
        = 2 \Re \langle f_\ell, f_{(I+\alpha J)^{-1}\ell} \rangle = 
        2\,\ell^T(I+\alpha J)^{-1}\ell          =
        2\,\ell^T
        \bigl(I+\alpha^2J^TJ\bigr)^{-1}\ell,
\]
where the last equality follows by considering
the symmetric part of the quadratic
form $(I + \alpha J)^{-1}$.
Taking the supremum over \(\| f_\ell\| = \|\ell\|_{\mathds{R}^d} = 1\), we obtain
\[
        \sigma_{\mathcal H_1}^2(J)
        =
        \frac{2}{1+\alpha^2\lambda_{\min}(J^TJ)} .
\]
Thus, for the class of centered linear observables, maximizing the local
singular-value gap \(\gamma_{\mathcal H_1}(J)\) is equivalent to minimizing
the worst-case asymptotic variance \(\sigma_{\mathcal H_1}^2(J)\).

Therefore, variance reduction for centered linear observables over $J$ amounts to maximizing $\lambda_{\min}(J^TJ)$.
If we impose the constraint in Frobenius norm
        $\|J\|_F=1$,
then for even \(d\) we get
\[
        \lambda_{\min}(J^TJ)
        \le
        \frac1d\operatorname{Tr}(J^TJ)
        =
        \frac1d\|J\|_F^2
        =
        \frac1d .
\]
The equality holds precisely when all eigenvalues of \(J^T J\) are equal, i.e.
when $J^TJ=\frac1d I$.
One such optimizer is
\[
        J_*
        =
        \frac1{\sqrt d}
        \begin{pmatrix}
        0&-1&&&&\\
        1&0&&&&\\
        &&0&-1&&\\
        &&1&0&&\\
        &&&&\ddots&
        \end{pmatrix},
\]
with \(d/2\) identical \(2\times2\) rotation blocks. For this choice,
\[
        \gamma_{\mathcal H_1}(J_*)
        =
        \sqrt{1+\frac{\alpha^2}{d}}, \qquad 
        \sigma_{\mathcal H_1}^2(J_*)
        =
        \frac{2}{1+\alpha^2/d}.
\]
By Theorem \ref{finite time variance bound by asymptotic variance}, for all \(g\in\mathcal H_1\),
\[
        \operatorname{Var}_\mu(\mu_tg)
        \le
        \left(
        \frac{2}{(1+\alpha^2/d)t}
        +
        \frac{4}{(1+\alpha^2/d)t^2}
        \right)
        \|g\|^2 .
\]

If \(d\) is odd, every real skew-symmetric \(J\) has a nontrivial kernel,
so $\lambda_{\min}(J^TJ)=0$. In this case,
no skew-symmetric perturbation $J$ can improve the worst-case empirical
or asymptotic variance over all centered linear observables:
\[
        \gamma_{\mathcal H_1}(J)=1,
        \qquad
        \sigma_{\mathcal H_1}^2(J)=2.
\]
This obstruction disappears if one restricts to the subspace $K$ of linear
observables $f_\ell$ with $\ell$ orthogonal to \(\ker J\). One may check 
that $K$ is invariant, so we can obtain variance reduction similar to 
the even-dimensional case.
\end{example}

\section{Relation to other mixing quantities} \label{relation to other notions section}

\subsection{Relation to the reversibilized process} \label{relation to symmetrized process subsection}

When we have a bounded Markov generator $L$, its symmetric part can be written as
\begin{align*}
    L_s \coloneqq \frac{L + L^*}{2},
\end{align*}
which is self-adjoint on $L^2(\mu)$. The Markov process with generator $L_s$ is reversible. We 
denote by $\gamma_s$ the spectral gap of $L_s$, which coincides with its singular-value gap.
If $L$ is in addition normal, $\gamma$ is the gap of the spectrum
of $L_0 = L \vert_{L^2_0(\mu)}$and $\gamma_s$ is the gap of the real 
part of the spectrum of $L_0$, so it is natural to expect that $\gamma_s$ lower bounds $\gamma$.

For unbounded \(L\), the operator expression \((L+L^*)/2\) may have domain
issues. In this case, we define $\gamma_s$ through the symmetric form associated with $L$:
\begin{align} \label{definition of gamma_s through symmetric form}
        \gamma_s :=
        \inf_{\substack{f\in\Dom(L)\cap L^2_0(\mu), \ \|f\|=1}}
        \operatorname{Re}\langle f,-Lf\rangle .
\end{align}
This agrees with the previous definition when \(L\) is bounded. Moreover, if
there is
\[
        \mathcal C\subset \Dom(L)\cap\Dom(L^*)
\]
that is a core for \(L\) in the graph norm with $\mathbf{1} \in \mathcal{C}$, then \(\gamma_s\) can be computed on
\(\mathcal C\) as
\[
        \gamma_s
        =
        \inf_{\substack{f\in\mathcal C\cap L^2_0(\mu),\ \|f\|=1}}
        \left\langle f,-\frac{L+L^*}{2}f\right\rangle .
\]
In the smooth diffusion examples considered in the present paper, this is the usual 
way to understand 
the symmetrized generator and compute the symmetrized gap.

\begin{proposition} \label{symmetrized gamma lower bound}
    Let $\gamma$ be the singular-value gap of $L$ and $\gamma_s$ be defined as in
    \eqref{definition of gamma_s through symmetric form}. Then
    \[
    \gamma_s \leq \gamma.
    \]
\end{proposition}
\begin{proof}
    For any $f \in \Dom(L) \cap L^2_0(\mu)$, we have
    \[
    \Re \langle f ,-Lf \rangle \leq  \left| \langle f, -Lf \rangle \right|
    \leq \|f\| \, \|Lf\|.
    \]
    Taking the infimum over $\|f\| = 1$, we get $\gamma_s \leq \gamma$.
\end{proof}

\begin{remark}
    In the discrete time case, let $\gamma_A$ be the singular-value gap of the chain with transition matrix
    $A = \frac{1}{2}(P+P^*)$, as considered in \cite{chatterjee}. Then the 
    same proof as above shows that $\gamma_A \leq \gamma$, improving the $\frac{1}{2}$ factor in 
    Theorem 1.3 of \cite{chatterjee}.
\end{remark}

It is natural to ask whether an upper bound of $\gamma$ could be also obtained through $\gamma_s$. Such 
a bound cannot hold in the most general setting, as demonstrated by Example \ref{first example BM with drift}. That is, diffusion on the unit circle with generator $L = \partial_\theta^2+ \alpha \partial_\theta$
has $\gamma_s = 1$ but $\gamma = \sqrt{1+\alpha ^2}$. From this counterexample it is natural to expect 
that if we demand the strength of the dynamics (in this case the drift) to be bounded, then $\gamma$ might be bounded above in terms of $\gamma_s$. 
We prove such a bound for jump processes with bounded jump rates.

Let $S$ be a measurable space. Following the convention of 
\cite{lawler-sokal}, we consider unnormalized jump
kernel \(q(x,dy)\): for each \(x\), \(q(x,\cdot)\)
is a finite measure on \(S\), and for each measurable \(A\), the map
\(x\mapsto q(x,A)\) is measurable. The total mass
\[
        q(x):=q(x,S)
\]
is the jump rate at \(x\). 
The interpretation is that \(q(x)\) is the exponential rate
at which a jump is proposed from state \(x\), and,
conditional on such a
jump, the next state is sampled from the normalized kernel
\[
        p(x,dy):=\frac{q(x,dy)}{q(x)}
        \qquad \text{if } q(x)>0.
\]
The generator for this process is
\begin{align} \label{jump process generator}
        Lf(x)=\int (f(y)-f(x))\,q(x,dy).
\end{align}
Since the jump rates are bounded and \(\mu\) is invariant, \(L\) is a
bounded operator on \(L^p(\mu)\) for every \(1\le p\le\infty\). Indeed,
for \(r\ge q_{\max}\), 
\[
        \Pi(x,dy)=\frac{q(x,dy)}{r}
        +\left(1-\frac{q(x)}{r}\right)\delta_x(dy)
\]
is a Markov kernel preserving \(\mu\), and $L=r(\Pi-I)$.
Therefore, \(\|\Pi\|_{L^p\to L^p}\le1\), hence
\[
        \|Lf\|_p\le 2r\|f\|_p \qquad \text{ for } \ 1 \leq p \leq \infty.
\]
In particular, \(\operatorname{Dom}_{L^2}(L)=L^2(\mu)\).

\begin{proposition} \label{symmetrized gap upper bound proposition}
Consider a jump process on a measurable space $S$ with generator \eqref{jump process generator}
and an invariant probability measure $\mu$.
Assume the jump rates are bounded:
\[
        q(x,S)\le q_{\max}<\infty
        \qquad \text{for \(\mu\)-a.e. }x .
\]
Let \(\gamma\) be the singular-value gap of \(L\) on \(L_0^2(\mu)\), and
let \(\gamma_s\) be the spectral gap of the symmetrized generator $L_s$,
which is a bounded self-adjoint operator.
Then
\[
        \gamma\le \sqrt{2q_{\max}\gamma_s}.
\]
\end{proposition}

\begin{proof}
Let
\[
        Q(dx,dy):=\mu(dx)q(x,dy)
\]
be the stationary jump flux. Since \(\mu\) is invariant, for every bounded
measurable \(h\),
\[
        \int Lh(x)\,\mu(dx)= 0.
\]
By truncation, the identity also holds for every $h \in L^1(\mu)$, so
\begin{align} \label{jump process stationary flux identity}
        \iint h(y)\,Q(dx,dy)
        =
        \iint h(x)\,Q(dx,dy).
\end{align}
For \(f\in L^2(\mu)\), the Dirichlet form of the symmetrized process is
$\langle f,-L_s f\rangle = \operatorname{Re}\langle f,-Lf\rangle$.
We claim that
\begin{align} \label{Dirichlet form of symmetrized process}
        \langle f,-L_s f\rangle
        =
        \frac12
        \iint |f(y)-f(x)|^2\,Q(dx,dy).
\end{align}
Indeed,
\[
\begin{aligned}
        2\operatorname{Re}\langle f,-Lf\rangle
        &=
        2\operatorname{Re}
        \int \overline{f(x)}
        \int (f(x)-f(y))\,q(x,dy)\,\mu(dx)  \\
        &=
        2\iint |f(x)|^2\,Q(dx,dy)
        -
        2\operatorname{Re}
        \iint \overline{f(x)}f(y)\,Q(dx,dy).
\end{aligned}
\]
Using \eqref{jump process stationary flux identity} with \(h=|f|^2\),
\[
        \iint |f(x)|^2\,Q(dx,dy)
        =
        \iint |f(y)|^2\,Q(dx,dy).
\]
Therefore
\[
\begin{aligned}
        2\operatorname{Re}\langle f,-Lf\rangle
        &=
        \iint
        \left(
        |f(x)|^2+|f(y)|^2
        -2\operatorname{Re}(\overline{f(x)}f(y))
        \right)\,Q(dx,dy)  \\
        &=
        \iint |f(y)-f(x)|^2\,Q(dx,dy).
\end{aligned}
\]
This proves the claimed identity \eqref{Dirichlet form of symmetrized process}.
Next, by Cauchy--Schwarz,
\[
\begin{aligned}
        |Lf(x)|^2
        &=
        \left|
        \int (f(y)-f(x))\,q(x,dy)
        \right|^2  \\
        &\le
        q(x,S)
        \int |f(y)-f(x)|^2\,q(x,dy)  \\
        &\le
        q_{\max}
        \int |f(y)-f(x)|^2\,q(x,dy).
\end{aligned}
\]
Integrating over \(x\) gives
\[
        \|Lf\|^2
        \le
        q_{\max}
        \iint |f(y)-f(x)|^2\,Q(dx,dy)
        =
        2q_{\max}\langle f,-L_s f\rangle .
\]
Taking a sequence \((f_n)_{n\geq 1}\) in  \(L^2_0(\mu)\) such that $\|f_n\| = 1$
and $\langle f_n,-L_s f_n\rangle\to \gamma_s$,
\[
        \gamma^2
        =
        \inf_{\substack{f\in L_0^2(\mu)\\ \|f\|=1}}
        \|Lf\|^2
        \le
        \liminf_{n\to\infty}\|Lf_n\|^2
        \le
        \liminf_{n\to\infty} \,2q_{\max}\langle f_n,-L_s f_n\rangle = 
        2q_{\max}\gamma_s,
\]
as desired.
\end{proof}

\subsection{Relation to the mixing time}
Next, we relate the singular-value gap $\gamma$ to the mixing time in total-variation distance. For a Markov process 
with an invariant probability measure $\mu$, 
write
\[
    d(t) \coloneqq  \sup_{\nu\in\mathcal P(S)}
        \|\nu P_t-\mu\|_{\mathrm{TV}},  \qquad \quad
    t_{\mathrm{mix}}(\varepsilon) \coloneqq 
\inf \{t \geq 0: d(t) \leq \varepsilon\},
\]
where $\mathcal P(S)$ is the set of probability measures on $S$. Equivalently,
\[
        d(t)=\sup_{x\in S}\|P_t(x,\cdot)-\mu\|_{\mathrm{TV}}.
\]
Indeed, the inequality ``\(\ge\)'' follows by taking \(\nu=\delta_x\). Conversely, for any measurable \(A\),
\[
(\nu P_t-\mu)(A)
=
\int (P_t(x,A)-\mu(A))\,\nu(dx).
\]
Therefore
\[
|(\nu P_t-\mu)(A)|
\le
\sup_{x\in S}|P_t(x,A)-\mu(A)|
\le
\sup_{x\in S}\|P_t(x,\cdot)-\mu\|_{\mathrm{TV}}.
\]
Taking the supremum over \(A\) proves the reverse inequality.

The above definitions are canonical for finite state space discrete-time
chains, where $t$ takes values in nonnegative integers and $P_t = P^t$. 
On a general state space, total variation is one of many metrics for mixing, and it is 
usually more meaningful when the process is uniformly ergodic in the sense that 
$d(t) \to 0$ as $t \to \infty$.

\subsubsection{Lower bound on $t_{\mathrm{mix}}(\varepsilon)$}

For a Markov kernel $\pi(x,dy)$ on a standard Borel space $S$, write 
\begin{align*}
    \delta(\pi) \coloneqq \sup_{x,x' \in S} \|\pi(x,\cdot) - \pi(x', \cdot)\|_{\mathrm{TV}}.
\end{align*}
Note that $\delta(P_t) \leq 2d(t)$ by triangular inequality.
We need the following lemma from \cite{sethuraman-varadhan}. The 
original paper considers real-valued functions, but the proof still goes through for complex-valued functions.
\begin{lemma}[Complex form of Lemma 4.1 in \cite{sethuraman-varadhan}] \label{Lemma 4.1 in sethuraman-varadhan}
    Let $S$ be a standard Borel space, $Q$ a probability measure on $S \times S$ with marginals $\alpha$ and $\beta$, and $\pi(x,dy)$ a Markov kernel satisfying $Q(dx,dy) =
    \alpha(dx)\pi(x,dy)$.
    If $f \in L^2(\alpha)$ and $g \in L^2(\beta)$ satisfy
    \[
        \int f \,d\alpha = \int g \, d\beta = 0,
    \]
    then
    \[
      \Big| \int \!\! \int f(x)g(y)  \, Q(dx,dy)    \Big| \leq \sqrt{\delta(\pi)}  \, \|f\|_{L^2(\alpha)}  \,
      \|g\|_{L^2(\beta)}
    \]
\end{lemma}

\begin{corollary} \label{delta contraction corollary}
    Let $\pi(x,dy)$ be a Markov kernel with an invariant probability measure 
    $\mu$. As usual, use $\|\cdot\|$ to denote the $L^2(\mu)$-norm. Then 
    for any function $g \in L^2_0(\mu)$,
    \begin{align*}
        \|\pi g \| \leq \sqrt{\delta(\pi)} \, \| g\|.
    \end{align*}
\end{corollary}
\begin{proof}
    Since $\mu$ is invariant, $\mu(\pi g) = \mu(g) = 0$.
    We can then apply Lemma \ref{Lemma 4.1 in sethuraman-varadhan} with $\alpha = \beta = \mu$ and 
    $f = \overline{\pi g}$ to get
    \begin{align*}
        \|\pi g\|^2 &= \int \overline{\pi g}(x)\, \pi g(x) \, \mu(dx) \\
                    &= \int \!\! \int \overline{\pi g}(x) \,g(y) \, \pi(x,dy) \mu(dx) \\
                    &\leq \sqrt{\delta(\pi)} \, \| \pi g\| \, \| g\|.
     \end{align*}
     If $\pi g = 0$, we are done; otherwise, divide both sides by $\| \pi g \|$.
\end{proof}

\begin{theorem} \label{mixing time lower bound}
    Consider either a continuous-time Markov process or a discrete-time Markov chain 
    defined on a standard Borel space, with singular-value gap $\gamma$ defined in 
    \eqref{gap definition continuous time} and \eqref{gap definition discrete time}
    respectively. 
    Let $\varepsilon \in (0, \frac{1}{2})$. Let $t > 0 $ be such that $d(t) \leq \varepsilon$. Then 
    $\gamma \geq \frac{1 - \sqrt{2\varepsilon}}{t}$. In particular, 
    \begin{align*}
        \frac{1-\sqrt{2\varepsilon}}{\gamma} \leq t_{\mathrm{mix}}(\varepsilon).
    \end{align*}
\end{theorem}
\begin{proof}
     Fix $\varepsilon \in (0, \frac{1}{2})$ and $f \in \Dom(L) \cap L^2_0(\mu)$.  The idea is that to obtain a lower bound on $\gamma$, we can lower bound $\|Lf\|$ by $\|f\|$ through 
    the intermediate quantity $\|f -P_tf\|$ and control $\|P_tf\|$ by contraction in total-variation distance,
    thanks to Corollary \ref{delta contraction corollary}.

    \underline{Continuous-time case}:
    Let $t > 0$ be such that $d(t) \leq \varepsilon$.
    Using the identity $f - P_t f = \int_0^t P_s(-Lf) \,ds$, we have
    \begin{align} \label{L2 bound 1 in mixing time lower bound}
        \|f - P_t f\| \leq \int_0^t \|P_s Lf\| \, ds \leq t \, \|Lf\|.
    \end{align}
    Recall that $\delta(P_t) \leq 2d(t) \leq 2\varepsilon$.
    Apply Corollary \ref{delta contraction corollary} with $\pi = P_t$ to get
    \[
        \|P_t f\| \leq \sqrt{\delta(P_t)} \, \|f\| \leq \sqrt{2\varepsilon} \, \|f\|
    \]
    Therefore,
    \begin{align} \label{L2 bound 2 in mixing time lower bound}
        \|f - P_t f\| \geq \|f\| - \|P_t f\| \geq (1 - \sqrt{2\varepsilon}) \, \|f\|.
    \end{align}
    Combining \eqref{L2 bound 1 in mixing time lower bound} and \eqref{L2 bound 2 in mixing time lower bound},
    we get 
    \begin{align*}
        \frac{\|Lf\|}{\|f\|} \geq \frac{1-\sqrt{2\varepsilon}}{t}.
    \end{align*}
    Taking infimum over $f \in L^2_0(\mu)$, we conclude $\gamma \geq \frac{1 - \sqrt{2\varepsilon}}{t}$.
    We get $t_{\mathrm{mix}}(\varepsilon) \geq \frac{1-\sqrt{2\varepsilon}}{\gamma}$ by taking 
    infimum over all $t$ such that $d(t) \leq \varepsilon$.

    \underline{Discrete-time case}:
    Let $n \geq 1$ be such that $d(n) \leq \varepsilon$.
    Applying Corollary \ref{delta contraction corollary} with $\pi = P^n$, the same argument 
    leading to \eqref{L2 bound 2 in mixing time lower bound} gives
    \begin{align} \label{L2 bound 3 in mixing time lower bound}
        \| f - P^n f\| \geq (1-\sqrt{2\varepsilon}) \, \|f\|.
    \end{align}
    On the other hand, with $L = P-I$,
    \begin{align} \label{L2 bound 4 in mixing time lower bound}
        \|f - P^n f \| \leq \sum_{k=0}^{n-1} \|P^{k+1} f - P^k f\| = \sum_{k=0}^{n-1} \|P^k (P-I)f\| \leq n \, \|Lf\|
    \end{align}
    Combining \eqref{L2 bound 3 in mixing time lower bound} and \eqref{L2 bound 4 in mixing time lower bound},
    we get
    \[
         \frac{\|Lf\|}{\|f\|} \geq \frac{1-\sqrt{2\varepsilon}}{n}.
    \]
    Taking infimum over $f \in L^2_0(\mu)$, we conclude $\gamma \geq \frac{1 - \sqrt{2\varepsilon}}{n}$.
    In particular, we can take $n = t_{\mathrm{mix}}(\varepsilon)$.
\end{proof}

\begin{remark}
    The discrete-time case above improves by a constant the lower bound on the mixing time in Theorem 1.4 of \cite{chatterjee}.
    In fact, a bound with a sharper constant than stated in Theorem 1.4 of \cite{chatterjee} could be derived from 
    an observation in the same paper:
    Let $\gamma_{ps}$ be the pseudo-spectral gap in \cite{paulin2018concentration}, defined by
    \[
    \gamma_{\mathrm{ps}} \coloneqq \max_{k \geq 1} \frac{\gamma((P^*)^k P^k)}{k},
    \]
    where $\gamma((P^*)^k P^k)$ is the spectral gap of the self-adjoint $(P^*)^k P^k$. Proposition 3.4
    in \cite{paulin2018concentration} shows that for $\varepsilon \in (0, \frac{1}{2})$, 
    $\frac{1-2\varepsilon}{\gamma_\mathrm{{ps}}} \leq t_{\mathrm{mix}}(\varepsilon)$. As observed in \cite{chatterjee},
    $\gamma \geq \frac{1}{2} \gamma_{\mathrm{ps}}$. These two estimates combine to give 
    \[
        \frac{1-2\varepsilon}{2 \gamma} \leq t_{\mathrm{mix}}(\varepsilon)
    \]
    in the discrete-time case, with a worse constant compared to what we have above.
\end{remark}

\subsubsection{Upper bound on $t_{\mathrm{mix}}(\varepsilon)$}
We can obtain an upper bound on the mixing time by $\gamma$ only in the restrictive 
setting of jump processes on finite state space. Note from the proof below that 
the upper bound can in fact be stated in terms of $\gamma_s$.
\begin{proposition} \label{mixing time upper bound finite state space}
Consider a jump process on a finite state
space \(S\) with an invariant probability measure \(\mu\);
the generator is given as in \eqref{jump process generator}. Assume $\gamma > 0$,
\[
        \mu_{\min}:=\min_{x\in S}\mu(x)>0,
\]
and that the jump rates are bounded by \(q_{\max}\). Then, for
\(\varepsilon\in(0,1/2)\),
\[
        t_{\mathrm{mix}}(\varepsilon)
        \le
        \frac{2q_{\max}}{\gamma^2}
        \log\left(\frac{1}{2\varepsilon \sqrt{\mu_{\min}}}\right).
\]
\end{proposition}

\begin{proof}
Let \(\gamma_s\) be the symmetrized gap on
\(L_0^2(\mu)\) defined in \eqref{definition of gamma_s through symmetric form}. 
We first note that \(\gamma_s\) controls the \(L^2(\mu)\)-decay of the
original semigroup, with a derivation standard in the reversible scenario.
If \(f\in L_0^2(\mu)\), then \(P_t f\in L_0^2(\mu)\)
because \(\mu\) is invariant. Therefore
\[
\begin{aligned}
        \frac{d}{dt}\|P_t f\|^2
        &=
        \langle LP_t f,P_t f\rangle
        +
        \langle P_t f,LP_t f\rangle  \\
        &=
        2\operatorname{Re}\langle P_t f,LP_t f\rangle  \\
        &\le
        -2\gamma_s\|P_t f\|^2.
\end{aligned}
\]
By Gronwall's inequality,
\begin{align} \label{L2 contraction by gamma_s}
        \|P_t f\|
        \le
        e^{-\gamma_s t}\|f\|.
\end{align}
Next, recall the dual characterization of total variation: for any $x \in S$,
\[
        \|P_t(x,\cdot)-\mu\|_{\mathrm{TV}}
        =
        \frac12
        \sup_{\|h\|_\infty\le1}
        |P_th(x)-\mu h|.
\]
For any bounded measurable \(h\), let $\bar h:=h-\mu h$.
Then \(\bar h\in L_0^2(\mu)\), and
\[
        P_th(x)-\mu h=P_t\bar h(x).
\]
Moreover, if \(\|h\|_\infty\le1\), then $\| \bar{h} \| \leq 1$.
For any $g \in L^2(\mu)$, 
\[
       \sqrt{\mu(x)} \, |g(x)| 
    \leq \Big( \sum_{y \in S} \mu(y) \, |g(y)|^2    \Big)^{1/2} = \| g\|
\]
Combining the above ingredients, we get for any $h$ on $S$ with $\|h\|_\infty \leq 1$,
\[
        |P_t h(x)-\mu h|
        =
        |P_t\bar h(x)|
        \le
        \frac{1}{\sqrt{\mu(x)}}e^{-\gamma_s t}\|\bar h\|
        \le
        \frac{1}{\sqrt{\mu_{\min}}}e^{-\gamma_s t}.
\]
Recall from Proposition \ref{symmetrized gap upper bound proposition} that
\[
        \gamma_s\ge \frac{\gamma^2}{2q_{\max}}.
\]
Hence
\[
        \|P_t(x,\cdot)-\mu\|_{\mathrm{TV}}
        \le
        \frac{1}{2\sqrt{\mu_{\min}}}e^{-\gamma_s t} \leq \frac{1}{2\sqrt{\mu_{\min}}} 
        \exp \Big( \frac{-\gamma^2 t}{2 q_{\mathrm{max}}}        \Big).
\]
In particular, for $t =  \frac{2q_{\max}}{\gamma^2}
        \log\left(\frac{1}{2\varepsilon \sqrt{\mu_{\min}}}\right)$, 
we have $ \|P_t(x,\cdot)-\mu\|_{\mathrm{TV}} \leq \varepsilon$.
\end{proof}

The finite state space assumption in Proposition \ref{mixing time upper bound finite state space} 
above is essential. On a countably
infinite state space, a positive spectral gap does not imply finite
worst-case total-variation mixing time, even for a reversible jump process
with bounded rates.

\begin{example} \label{MM1 example}
Consider the continuous-time \(M/M/1\) queue, or reflected biased random walk, on \(\mathds N=\{0,1,2,\ldots\}\). Fix \(0<\lambda<\mu\), and let
\[
        q(n,n+1)=\lambda,\quad n\ge0 \qquad \text{and} \qquad
        q(n,n-1)=\mu,\quad n\ge1.
\]
The generator of this process is given by
\begin{align*}
        (Lf)(0)&=\lambda(f(1)-f(0)), \\
        (Lf)(n)
        &=
        \lambda(f(n+1)-f(n))
        +
        \mu(f(n-1)-f(n)), \quad n \geq 1.
\end{align*}
The total jump rate is bounded by \(\lambda+\mu\). The invariant measure is
the geometric distribution
\[
        \pi(n)=(1-\rho)\rho^n,
        \qquad
        \rho=\frac{\lambda}{\mu}<1.
\]
Moreover, $\pi(n)\lambda=\pi(n+1)\mu$,
so the process is reversible.
This process has a positive \(L^2(\pi)\) spectral gap. In fact, we can explicitly compute
the spectral gap to be (see Appendix \ref{appendix Jacobi spectrum})
\[
        \gamma
        =
        (\sqrt{\mu}-\sqrt{\lambda})^2.
\]

Nevertheless, the worst-case total-variation distance satisfies
\[
        d(t):=\sup_{n\in\mathds N}
        \|P_t(n,\cdot)-\pi\|_{\mathrm{TV}}
        =
        1
\]
for every finite \(t\). To see this, fix \(\eta>0\), and choose \(R\) so
large that
\[
        \pi(\{0,1,\ldots,R\})>1-\eta.
\]
If the chain starts from \(n>R\), then it must make at least \(n-R\)
downward jumps before it can enter \(\{0,1,\ldots,R\}\). The total number
of jumps by time \(t\) is stochastically dominated by a Poisson random
variable with mean \((\lambda+\mu)t\). Hence
\[
        P_t(n,\{0,1,\ldots,R\})
        \le
        \mathds P\!\left(
        \mathrm{Pois}((\lambda+\mu)t)\ge n-R
        \right)
        \longrightarrow 0
\]
as \(n\to\infty\). Therefore
\[
\begin{aligned}
        \sup_{n\in\mathds N}
        \|P_t(n,\cdot)-\pi\|_{\mathrm{TV}}
        &\ge
        \limsup_{n\to\infty}
        \left|
        P_t(n,\{0,1,\ldots,R\})
        -
        \pi(\{0,1,\ldots,R\})
        \right|  \\
        &\ge
        1-\eta.
\end{aligned}
\]
Taking \(\eta\downarrow0\), we get \(d(t)=1\) for every finite \(t\).
Thus,
\[
        t_{\mathrm{mix}}(\varepsilon)=\infty
\]
for every \(\varepsilon<1\), despite the positive spectral gap.

The discrete-time version of this jump process was discussed in e.g. \cite{KonMey2003}
as an example Markov chain that is geometrically ergodic but not Doeblin recurrent.
For our purposes, this process shows that an \(L^2(\pi)\)-spectral gap 
cannot control worst-case total-variation mixing when the process has finite 
``speed" while the invariant measure has unbounded support.
\end{example}

\subsection{Relation to the Cheeger constant}

Consider a stationary Markov jump process on a standard Borel space $S$ with jump 
kernel $q(x,dy)$ and invariant probability measure $\mu$. Let the jump rate be bounded by $q_{\mathrm{max}}
\coloneqq \mathrm{ess \ sup}_x q(x,S) < \infty$. 
Recall that the generator is 
\begin{align} \label{generator of jump process 2}
        Lf(x) = \int \big(f(y)-f(x) \big) \,q(x,dy).
\end{align}
Let
\[
    Q(dx,dy) \coloneqq \mu(dx) q(x,dy)
\]
be the stationary flux of the jump process. The Cheeger constant is defined by 
\begin{align} \label{Cheeger constant definition}
    \xi \coloneqq \inf_{A \subset S:\, 0 < \mu(A) \leq \frac{1}{2}} \frac{Q(A, A^c)}{\mu(A)}
\end{align}

The constant $\xi$ is also known as the bottleneck ratio, with several slightly different 
definitions in the literature. 
When the jump process is reversible i.e. $L$ is self-adjoint in $L^2(\mu)$, 
we recall the classical result in 
\cite{lawler-sokal}, Theorems 2.1 and 3.5:
\begin{align} \label{lawler sokal cheeger inequality}
\frac{\xi^2}{2q_{\max}} \leq \gamma \leq 2 \xi,
\end{align}
where $\gamma$ is the usual spectral gap. In the general nonreversible setting,
we obtain 

\begin{theorem} \label{Cheeger theorem}
    Consider a Markov jump process on a standard Borel space with generator \eqref{generator of jump process 2}, an invariant probability measure $\mu$,
    and jump rates bounded by $q(x,S) \leq q_{\mathrm{max}} < \infty$ for $\mu$-a.e. $x$. Let $\gamma$ be the associated 
    singular-value gap and $\xi$ the Cheeger constant defined in \eqref{Cheeger constant definition}. Then
    \begin{align*}
        \frac{\xi^2}{2q_{\mathrm{max}}} \leq \gamma \leq \frac{36}{5} \xi.
    \end{align*}
\end{theorem}

Before proving Theorem \ref{Cheeger theorem}, we first record some basic characterizations of the adjoint 
jump process.
Define the reversed flux measure by 
\[
        Q^\dagger(dx,dy):=Q(dy,dx).
\]
Since \(\mu\) is invariant, for every measurable set \(A\),
\begin{align} \label{equality Q(A,S) = Q(S,A)}
        0=\int L\mathbf 1_A\,d\mu
        =
        Q(S,A)-Q(A,S).
\end{align}
Thus, we have the equality of marginals
\[
       q(x) \, \mu(dx) =  Q(dx,S) = Q(S,dx) = Q^\dagger(dx,S).
\]
In particular,
since the first marginal of $Q^\dagger$ is absolutely continuous with respect to $\mu$ and
$S$ is standard Borel, there exists a finite kernel $q^\dagger(x,dy)$ such that 
\[
Q^\dagger(dx, dy) = \mu(dx) \, q^\dagger (x,dy). 
\]
Unsurprisingly, time reversal corresponds to $L^2(\mu)$-adjoint:
\begin{lemma} \label{Cheeger reverse lemma}
    Let $L^*$ be the adjoint of $L$ in $L^2(\mu)$. Then $q^\dagger(x,dy)$ is the jump kernel of the 
    adjoint process with generator $L^*$.
    Consequently, the reversed flux measure $Q^\dagger(dx,dy)$ is the stationary flux of the adjoint process.
\end{lemma}
\begin{proof}
The process with jump kernel $q^\dagger(x,dy)$ has generator
\[
L^\dagger g \,(x) \coloneqq \int \big( g(y) - g(x)   \big) \, q^\dagger(x,dy).
\]
We claim $L^\dagger = L^*$. Indeed, for any $f,g \in L^2(\mu)$,
\begin{align*}
        \langle g,Lf\rangle
        &=
        \iint \overline{g(x)}(f(y)-f(x))\,Q(dx,dy)  \\
        &=
        \iint \overline{g(y)}f(x)\,Q^\dagger(dx,dy)
        -
        \int \overline{g(x)}f(x)\,Q(dx,S) \\
        &=
        \iint \overline{g(y)}f(x)\,Q^\dagger(dx,dy)
        -
        \int \overline{g(x)}f(x)\,Q^\dagger(dx,S) \\
        &=
        \iint
        \bigl(\overline{g(y)}-\overline{g(x)}\bigr)f(x)
        \,Q^\dagger(dx,dy) \\
        &=
        \int
        \overline{
        \int (g(y)-g(x))q^\dagger(x,dy)
        }
        f(x)\,\mu(dx) \\
        &=
        \langle L^\dagger g,f\rangle .
\end{align*}
This shows that $q^\dagger(x,dy)$ and $Q^\dagger(dx,dy)$ are, respectively, the jump kernel 
and the stationary flux of the adjoint process.
\end{proof}

We are now ready to prove Theorem \ref{Cheeger theorem}. 
Note that the content of the lower bound is essentially 
in the proof of \cite{lawler-sokal}. The proof of the upper bound
refines the proof idea in Theorem 1.6 of \cite{chatterjee} 
and crucially depends on the connection between the singular-value gap
and the empirical variance of observables.

\begin{proof} [Proof of the lower bound in Theorem \ref{Cheeger theorem}]
Consider the reversibilized jump process with generator $L_s = \frac{L+L^*}{2}$.
Thanks to Lemma \ref{Cheeger reverse lemma}, we see that $L_s$ has stationary flux
\[
        Q_s(dx,dy)
        =
        \frac12\bigl(Q(dx,dy)+Q^\dagger(dx,dy)\bigr)
        =
        \frac12\bigl(Q(dx,dy)+Q(dy,dx)\bigr).
\]
Moreover, recall from \eqref{equality Q(A,S) = Q(S,A)} that for any measurable set $A$,
$Q(S,A) = Q(A,S)$. Since $Q(S,A) = Q(A,A) + Q(A^c,A)$ and $Q(A,S) = Q(A,A) + Q(A,A^c)$, we
deduce that 
\[
Q(A,A^c) = Q(A^c,A).
\]
Consequently,
\[
Q_s(A,A^c) = \frac{1}{2} \big( Q(A,A^c) + Q(A^c,A) \big) = Q(A,A^c).
\]
Therefore, by the definition \eqref{Cheeger constant definition},
the Cheeger constant \(\xi_s\) of the reversibilized jump process is
equal to the Cheeger constant of the original process:
\[
        \xi_s=\xi.
\]
Let $\gamma_s$ be the spectral gap of the reversibilized process, satisfying 
$\gamma \geq \gamma_s$ by Proposition \ref{symmetrized gamma lower bound}.
Since $Q_s(dx,S) = Q(dx,S)$, the jump rates of the reversibilized process are also bounded by \(q_{\max}\).
By the Cheeger lower bound for reversible jump processes in \eqref{lawler sokal cheeger inequality}, 
we conclude
\[
        \gamma \geq \gamma_s
        \ge
        \frac{\xi_s^2}{2q_{\max}}
        =
        \frac{\xi^2}{2q_{\max}}.
\]
\end{proof}

\begin{proof}[Proof of the upper bound in Theorem \ref{Cheeger theorem}]
It suffices to prove the estimate for an arbitrary measurable set
\(A\subset S\) with \(0<\mu(A)\le 1/2\), and then take the infimum over
such \(A\). Write
\[
        a:=\mu(A),
        \qquad
        c:=1-a,
        \qquad
        \zeta(A):=\frac{Q(A,A^c)}{\mu(A)} .
\]
If $\gamma = 0$, there is nothing to prove;
if \(\zeta(A)=0\), then the argument below in particular gives \(\gamma=0\). Thus, assume \(\gamma, \,\zeta(A)>0\) and consider
\[
        g_A:=\frac{\mathbf 1_A-a}{\sqrt{a(1-a)}}.
\]
Then \(g_A\in L_0^2(\mu)\) and \(\|g_A\|=1\). We first obtain a lower bound
on the stationary covariance of \(g_A\). Let
\[
        p_s:=\mathds P_\mu(X_0\in A,\ X_s\in A^c).
\]
Since \(X_0\sim\mu\),
\[
        \mathds P_\mu(X_0\in A,\ X_s\in A)=a-p_s.
\]
Therefore,
\[
        \mathds E_\mu[g_A(X_0)g_A(X_s)]
        =
        \frac{\mathds P_\mu(X_0\in A,\ X_s\in A)-a^2}{a(1-a)}  
        =
        1-\frac{p_s}{a(1-a)} .
\]
The event \(\{X_0\in A,\ X_s\in A^c\}\) can occur only if the process makes
at least one jump from \(A\) to \(A^c\) during the time interval \([0,s]\).
Let \(N_s(A,A^c)\) denote the number of such jumps. By Markov's inequality,
\[
        p_s\le \mathds E_\mu N_s(A,A^c).
\]
Since the process is stationary, the expected number of jumps from \(A\) to
\(A^c\) in time \(s\) is
\[
        \mathds E_\mu N_s(A,A^c)
        =
        s\,Q(A,A^c).
\]
Hence
\[
        \mathds E_\mu[g_A(X_0)g_A(X_s)]
        \ge
        1-\frac{s\,Q(A,A^c)}{a(1-a)}
        =
        1-\frac{\zeta(A)s}{c}.
\]
Using the stationary covariance identity \eqref{variance of empirical average 2}, we get
\[
\begin{aligned}
        \Var_\mu(\mu_t g_A)
        &=
        \frac{2}{t^2}
        \int_0^t (t-s)
        \mathds E_\mu[g_A(X_0)g_A(X_s)]\,ds  \\
        &\ge
        \frac{2}{t^2}
        \int_0^t (t-s)
        \left(1-\frac{\zeta(A)s}{c}\right)\,ds  \\
        &=
        1-\frac{\zeta(A)t}{3c}.
\end{aligned}
\]

On the other hand, by the upper bounds in \eqref{uniform upper bound of asymptotic variance by 2/gamma}
\eqref{bound using asymptotic variance and gap, observable specific},
\[
        \Var_\mu(\mu_t g_A)
        \le
        \frac{2}{\gamma t}
        +
        \frac{4}{\gamma^2t^2}.
\]
We now choose \(t\) optimally. Write
\[
        x:=\frac{\gamma}{\zeta(A)},
        \qquad
        t=\frac{rc}{\zeta(A)},
        \qquad r\in(0,3).
\]
Then the lower and upper bounds imply
\[
        1-\frac r3
        \le
        \frac{2}{xcr}
        +
        \frac{4}{x^2c^2r^2}.
\]
Let \(y:=cx\). Since $c$ is fixed for a fixed set $A$, minimizing 
the constant in the upper bound amounts to minimizing $y$.
For fixed \(r\), the inequality above gives
\[
        y
        \le
        F(r):=
        \frac{1+\sqrt{1+4(1-r/3)}}{r(1-r/3)}.
\]
A direct minimization shows that \(F\) is minimized at \(r=5/3\), where
\[
        F(5/3)=\frac{18}{5}.
\]
Indeed, setting \(u=1-r/3\) and \(s=\sqrt{1+4u}\), one obtains
\[
        F(r)=\frac{16}{3(s-1)(5-s^2)},
\]
so minimizing \(F\) is equivalent to maximizing
\((s-1)(5-s^2)\), whose unique critical point and local maximum in \((1,\sqrt5)\) is at
\(s=5/3\), which corresponds to \(r=5/3\).
Therefore, with
\[
        t=\frac{5c}{3\zeta(A)},
\]
we get
\[
        \frac{\gamma}{\zeta(A)}
        \le
        \frac{18}{5c}
        \le
        \frac{36}{5},
\]
since \(c\ge1/2\). Hence $\gamma\le \frac{36}{5}\zeta(A)$.
Taking the infimum over \(A\) proves
\[
        \gamma\le \frac{36}{5}\xi .
\]
\end{proof}

\begin{remark}
    The constant $1/2$ in the lower bound above improves the one stated in \cite{chatterjee}
    for finite-state discrete-time chains
    and matches the constant 
    in the reversible case. We already improved the constant in the upper bound from $32$ to $36/5$ by 
    refining Chatterjee's proof strategy, but perhaps it could be further optimized.
\end{remark}

\section{Applications} \label{applications section}

\subsection{Shift-register chain on binary de Bruijn graph}
Let
\[
        B_n=\{-1,1\}^n
\]
Define a discrete-time Markov chain on $B_n$ such that
from
\[
        x=(x_1,\ldots,x_n)
\]
the chain moves to
\[
        (x_2,\ldots,x_n,+1)
        \quad\text{or}\quad
        (x_2,\ldots,x_n,-1)
\]
with probability \(1/2\) each. 
This is the simple random walk on the directed binary de Bruijn graph.
The chain is nonreversible and has uniform stationary measure \(\mu\). Indeed, after \(n\) steps, all original coordinates have been forgotten, so
\[
        P^n(x,\cdot)=\mu
        \qquad\text{for every }x\in B_n.
\]
On the other hand, for \(k<n\), the law \(P^k(x,\cdot)\) is uniform on only \(2^k\) states, while \(\mu\) is uniform on \(2^n\) states. Hence
\[
        \|P^k(x,\cdot)-\mu\|_{\mathrm{TV}}
        =
        1-2^{k-n},
        \qquad 0\le k<n.
\]
In particular, for every \(\varepsilon<1/2\),
\[
        t_{\mathrm{mix}}(\varepsilon)=n.
\]

We now compute the singular-value gap of \(L=P-I\). Let
\[
        \chi_A(x):=\prod_{i\in A}x_i,
        \qquad A\subseteq\{1,\ldots,n\},
\]
be the Walsh characters. These form an orthonormal basis of \(L^2(\mu)\), where 
the constant subspace is spanned by
\(\chi_\emptyset=1\); see section 1.3 of \cite{ODonnell_AnalysisOfBooleanFunctions}. 
For \(A\neq\emptyset\),
\[
        P\chi_A
        =
        \begin{cases}
        \chi_{A+1}, & n\notin A,\\
        0, & n\in A,
        \end{cases}
\]
where
\[
        A+1:=\{i+1:i\in A, \, i < n\}.
\]
Thus \(P\) decomposes \(L^2_0(\mu)\) into orthogonal nilpotent shift chains. The longest such chain has length \(n\), for example:
\[
        \chi_{\{1\}}\mapsto\chi_{\{2\}}\mapsto\cdots\mapsto\chi_{\{n\}}\mapsto0.
\]
On a shift chain of length \(\ell\), \(P\) is the nilpotent shift
\begin{align} \label{shift chain}
        S_\ell: \ e_1\mapsto e_2\mapsto\cdots\mapsto e_\ell\mapsto0.
\end{align}
The smallest singular value of \(I-P\) on this \(\ell\)-dimensional subspace is
the square root of the smallest eigenvalue of the tridiagonal 
matrix $(I-S_\ell)^*(I-S_\ell)$:
\[
        2\sin\left(\frac{\pi}{4\ell+2}\right) \asymp \frac{1}{\ell}.
\]
Therefore, the global singular-value gap is
\[
        \gamma_n
        =
        2\sin\left(\frac{\pi}{4n+2}\right)
        \asymp \frac1n.
\]
Using Theorem \ref{mixing time lower bound}, we get
\[
        t_{\mathrm{mix}}(\varepsilon)
        \ge
        \frac{1-\sqrt{2\varepsilon}}{\gamma_n}
        \asymp n,
\]
which has the correct order.

This example is prototypically nonnormal and illustrates a slow-mixing mechanism
caused by deterministic or nearly deterministic memory that is not captured 
by traditional spectral quantities.
On $L^2_0(\mu)$, the transition operator is nilpotent, so all eigenvalues of \(P\)
are \(0\), hence all eigenvalues of \(L=P-I\) are \(-1\). Thus, the distance between zero
and the spectrum of $L$ is \(1\). One can also check that the symmetrized gap $\gamma_s$ has order \(n^{-2}\). Indeed, $-L_s$ acts on a shift block of length $\ell$ as
\[
    I - \frac{S_\ell +S^*_\ell}{2},
\]
with smallest eigenvalue 
\[
1- \cos \frac{\pi}{\ell + 1} \asymp \frac{1}{\ell^2}.
\]

By contrast, the singular-value gap of \(L\) is
of order \(n^{-1}\), exactly reflecting the length of the longest nilpotent
shift chain and giving the correct order for the mixing-time lower bound. In this case,
the reason that the singular-value gap ``sees'' the length of the shift chains
can be understood by the expansion
\[
(I - S_\ell)^{-1} = I + S_\ell + \cdots + S_\ell^{\ell-1},
\]
from which we see that $\|(I-S_\ell)^{-1}\|$, which is the inverse of 
the local singular-value gap by Lemma \ref{main invertible lemma}, has order $\ell$.

\subsection{A hypoelliptic shear on the torus}
We first observe a mechanism for
a Markov generator 
to possess a positive singular-value gap.

\begin{proposition}[Compact-resolvent criterion] 
\label{compact resolvent criterion proposition}
Let \(L\) be a
Markov generator on \(L^2(\mu)\), where $\mu$ is an invariant probability measure
on a measurable space.
Suppose that \(L\) has
compact resolvent. Equivalently, the embedding of $(\Dom(L), \| \cdot \| + \|L \cdot \|)$
into $L^2(\mu)$ is compact. 
Suppose further that the kernel of \(L\) consists only
of constants. Then \(L\) has a positive singular-value gap on
\(L^2_0(\mu)\):
\[
        \gamma
        :=
        \inf_{\substack{f\in\operatorname{Dom}(L)\cap L^2_0(\mu) \backslash \{0\}}}
        \frac{\|Lf\|}{\|f\|}
        >0.
\]
Consequently, the empirical-variance bounds of Corollary \ref{main global corollary} apply.
\end{proposition}
\begin{proof}
    Since $\mu$ is invariant, $L^2_0(\mu)$ is invariant under the semigroup
    and hence under the resolvent. Therefore, the restriction $L_0$ of $L$ to $L^2_0(\mu)$ also has compact resolvent.
    For an operator with compact resolvent, the spectrum consists only of
isolated eigenvalues of finite algebraic multiplicity, with no finite
accumulation point. Since the kernel of $L$ consists only of constants,
$0$ lies in the resolvent set of $L_0$. Therefore, $L_0: L^2_0(\mu) \cap \Dom(L) \to
L^2_0(\mu)$ is bijective with bounded inverse, hence $\gamma > 0$.
\end{proof}

Hypoelliptic diffusion generators often satisfy the assumptions of the previous 
proposition. We illustrate this point with the following example. 
Let \(M=\mathds T^2\) and consider
\[
        L=\partial_y^2+\sin(y)\partial_x .
\]
An invariant measure $\mu$ is the normalized Lebesgue measure. The operator is nonreversible:
\(\partial_y^2\) is symmetric and \(\sin(y)\partial_x\) is skew-adjoint in
\(L^2(\mu)\).
Unlike the OU process
in section \ref{OU process section}, the spectrum of 
$L$ is not explicit.

Nevertheless, we can check that $L$ has a positive
singular-value gap on \(L^2_0(\mu)\) by checking 
the conditions of Proposition \ref{compact resolvent criterion proposition}.
We first note that $L$ satisfies
Hörmander's bracket condition \cite{Hormander1967}. The
relevant vector fields are
\[
        X_1=\partial_y,\qquad X_0=\sin(y)\partial_x.
\]
Their commutator is
\[
        [X_1,X_0]=\cos(y)\partial_x.
\]
Since \(\sin(y)\) and \(\cos(y)\) do not vanish simultaneously, the vector
fields \(X_1,X_0,[X_1,X_0]\) span the tangent space at every point of
\(\mathds T^2\). Therefore, Hörmander's bracket condition holds, and \(L\) is
hypoelliptic.

We now identify the kernel of $L$. If \(Lf=0\), then hypoellipticity
implies that \(f\) is smooth. Taking real parts in the identity
\(\langle f,Lf\rangle=0\), we obtain
\[
        0=\operatorname{Re}\langle f,Lf\rangle
        =
        -\|\partial_y f\|^2.
\]
Thus, \(f\) is independent of \(y\), say \(f(x,y)=F(x)\). The equation
\(Lf=0\) then gives
\[
        \sin(y)F'(x)=0,
\]
hence \(F\) is constant. Thus, the kernel of \(L\) consists only of
constants.

One may also check that $L$ has compact resolvent by showing that the embedding
$(\Dom(L), \|\cdot\| + \|L \cdot \|) \hookrightarrow L^2(\mathds{T}^2)$ is compact.
See Appendix \ref{compactness proof appendix} for details. As a result,
conditions for Proposition \ref{compact resolvent criterion proposition} are 
satisfied and
\[
        \gamma
        =
        \inf_{\substack{f\in\operatorname{Dom}(L)\cap L^2_0(\mu) \backslash \{0\}}}
        \frac{\|Lf\|}{\|f\|}
        >0.
\]
Consequently, Corollary \ref{main global corollary} yields
\[
        \Var_\mu (\mu_t g)
        \le
        \frac{4}{\gamma t}\|g\|^2,
        \qquad g\in L^2_0(\mu).
\]

Although the spectrum of $L$ is not explicit in this example, we
take advantage of the standard tools for hypoelliptic differential operators
and the compact-resolvent criterion in Proposition \ref{compact resolvent criterion proposition}
to show that $\gamma > 0$.
Thus, even though the symmetric part \(\partial_y^2\) has no spectral gap on
\(L_0^2(\mu)\), the coupling created by the shear drift $ \, \sin(y) \partial_x$ is
enough to give a uniform stationary empirical-variance bound.

More generally, the compact-resolvent assumption of Proposition \ref{compact resolvent criterion proposition}
is natural for hypoelliptic operators on
compact manifolds. Hörmander's bracket condition
can be checked to prove hypoellipticity for operators of the form \(L = X_0+\sum_i X_i^2\), and subelliptic estimates often
imply compactness of the embedding of $\Dom(L)$ into \(L^2\).

\newpage
\begin{appendix}
\section{Omitted details}
\subsection{The Poisson-equation formula for asymptotic variance} 
We record a derivation of the asymptotic
variance formula \eqref{asymptotic variance representation with L inverse}. 
This formula is known in the literature, but often with extra hypotheses, such 
as when a central limit theorem is sought in addition 
to a limiting second-moment. Thus,
we provide a clear statement
and proof in accordance with our operator-theoretic perspective in this paper.

\begin{lemma}[Poisson equation formula for asymptotic variance]
\label{appendix lemma proof of asymptotic variance formula}
Consider a strongly continuous Markov semigroup $(P_t)_{t \geq 0}$ on 
a measurable space with an invariant probability measure $\mu$.
Let \(K\subset L_0^2(\mu)\) be a closed \(P_t\)-invariant subspace, and
let \(L_K\) denote the generator of the restricted semigroup
\((P_t \vert_K)_{t\ge0}\).
Fix \(g\in K\). Suppose there exists \(f\in \operatorname{Dom}(L_K)\) such
that
\[
        -L_K f=g.
\]
Then the asymptotic variance of \(g\) exists and is given by
\[
        \sigma_g^2
        :=
        \lim_{t\to\infty}
        t\,\operatorname{Var}_\mu(\mu_tg)
        =
        2\operatorname{Re}\langle g,f\rangle .
\]
In particular, if the local gap defined in \eqref{gap definition invariant subspace}
satisfies $\gamma_K > 0$, then
\[
        \sigma_g^2
        =
        2\operatorname{Re}\langle g,-L_K^{-1}g\rangle .
\]
\end{lemma}

\begin{proof}
To simplify the notation, denote by $(P_t)_{t \geq 0}$ the semigroup restricted to $K$, and 
let $A_t$ be the corresponding restricted Cesàro-mean operator, defined in \eqref{Cesaro mean definition}.
By \eqref{variance of empirical average 2},
\[
        \operatorname{Var}_\mu(\mu_tg)
        =
        \frac{2}{t^2}\operatorname{Re}
        \int_0^t (t-s)\langle g,P_sg\rangle\,ds .
\]
Since \(f\in\operatorname{Dom}(L_K)\) and \(-L_Kf=g\), the map
\(s\mapsto P_s f\) is differentiable in \(L^2(\mu)\), and
\[
        \frac{d}{ds}P_s f
        =
        P_sL_Kf
        =
        -P_sg .
\]
Therefore
\[
        \frac{d}{ds}\langle g,P_s f\rangle
        =
        \langle g,-P_sg\rangle
        =
        -\langle g,P_sg\rangle .
\]
Thus
\[
\begin{aligned}
        \int_0^t (t-s)\langle g,P_sg\rangle\,ds
        &=
        -\int_0^t (t-s)
        \frac{d}{ds}\langle g,P_s f\rangle\,ds  
        =
        t\langle g,f\rangle
        -
        \int_0^t \langle g,P_s f\rangle\,ds .
\end{aligned}
\]
Multiplying by \(2/t\) and using the definition of \(A_t\), we obtain
\begin{align} \label{intermediate identity in proof of lemma 5}
        t\,\operatorname{Var}_\mu(\mu_tg)
        =
        2\operatorname{Re}\langle g,f\rangle
        -
        2\operatorname{Re}\langle g,A_t f\rangle .
\end{align}

It remains to show that as $t \to \infty$,
\[
        \langle g,A_t f\rangle\to0 .
\]
By Theorem V.4.5 and Example V.4.7 in \cite{engel_nagel_2000}, $(P_t)_{t \geq 0}$ is mean ergodic, so
the strong limit
\[
        \Pi h :=\lim_{t\to\infty}A_t h
\]
exists for all $h \in K$. By Lemma V.4.4 in \cite{engel_nagel_2000},
$\Pi$ is the mean-ergodic orthogonal projection onto $\mathrm{Ran}\, \Pi = \ker L_K$. Since 
$g = -L_K  f$, by \eqref{A_tg identity in terms of f} we have
\[
\|A_t g\| = \frac{\|f - P_t f\|}{t} \leq \frac{2}{t}\|f\| \to 0 \qquad \text{as} \quad t \to \infty.
\]
Therefore, \(\Pi g=0\), hence \(g\perp \operatorname{Ran}\Pi\). In particular,
\[
      \langle g,A_t f\rangle \to  \langle g,\Pi f\rangle=0.
\]
Returning to \eqref{intermediate identity in proof of lemma 5}, we conclude that
\[
        \sigma_g^2 = \lim_{t\to\infty}
        t\,\operatorname{Var}_\mu(\mu_tg)
        =
        2\operatorname{Re}\langle g,f\rangle .
\]

Finally, when $\gamma_K > 0$, $L_K$ is invertible by Lemma \ref{main invertible lemma},
so we may take $f = -L_K^{-1} g$.
\end{proof}

\subsection{Eigenfunctions of the rotating Ornstein--Uhlenbeck generator}
\label{Laguerre appendix}
In polar coordinates, the generator can be decomposed as
\[
    L = S + \alpha A, \quad \text{where} \quad
        S=\Delta-x\cdot\nabla
        =
        \partial_{rr}
        +
        \left(\frac1r-r\right)\partial_r
        +
        \frac1{r^2}\partial_{\theta\theta}, \quad A = \partial_\theta
\]
We look for eigenfunctions of the form
\[
        \psi(r,\theta)=R(r)e^{im\theta},
        \qquad m\in\mathds Z.
\]
Then \(\partial_\theta\psi=im\psi\), and \(S\psi=-\lambda\psi\) becomes
\[
        R''(r)
        +
        \left(\frac1r-r\right)R'(r)
        -
        \frac{m^2}{r^2}R(r)
        =
        -\lambda R(r).
\]
Let \(a=|m|\), \(\rho=r^2/2\), and set
\[
        R(r)=r^a Q(\rho).
\]
A direct calculation gives
\begin{align} \label{Laguerre equation}
        \rho Q''(\rho)+(a+1-\rho)Q'(\rho)
        +
        \frac{\lambda-a}{2}Q(\rho)=0.
\end{align}
This is the generalized Laguerre equation. Polynomial solutions occur
when
\[
        \frac{\lambda-a}{2}=k,
        \qquad k=0,1,2,\ldots,
\]
so by, say, p.781 of \cite{abramowitz1972handbook}, the equation \eqref{Laguerre equation}
is solved by
\[
        \lambda=2k+a,
        \qquad
        Q(\rho)=L_k^{(a)}(\rho).
\]
Thus
\[
        \psi_{k,m}(r,\theta)
        =
        c_{k,m}r^{|m|}L_k^{(|m|)}(r^2/2)e^{im\theta}
\]
satisfies
\[
        S\psi_{k,m}=-(2k+|m|)\psi_{k,m},
        \qquad
        \partial_\theta\psi_{k,m}=im\psi_{k,m}.
\]
The normalization constant $c_{k,m}$ can be computed from the orthogonality
relation
\[
        \int_0^\infty
        \rho^a e^{-\rho}
        L_k^{(a)}(\rho)L_\ell^{(a)}(\rho)\,d\rho
        =
        \frac{\Gamma(k+a+1)}{k!}\mathbf 1_{k=\ell}.
\]
Since \(d\mu=(2\pi)^{-1}e^{-r^2/2}r\,dr\,d\theta\), this gives
\[
        c_{k,m}
        =
        \left(
        \frac{k!}{2^{|m|}\Gamma(k+|m|+1)}
        \right)^{1/2}.
\]

\subsection{Spectrum of the Jacobi operator for the reflected birth-death chain}
\label{appendix Jacobi spectrum}
We consider the jump process in Example \ref{MM1 example} and compute its spectral gap.
Recall that the generator of this process is given by
\begin{align*}
        (Lf)(0)&=\lambda(f(1)-f(0)), \\
        (Lf)(n)
        &=
        \lambda(f(n+1)-f(n))
        +
        \mu(f(n-1)-f(n)), \quad n \geq 1.
\end{align*}
The invariant measure is given by 
 $\pi(n)=(1-\rho)\rho^n,
        \
        \rho=\frac{\lambda}{\mu}<1.$
Conjugating \(-L\) from \(L^2(\pi)\) to \(\ell^2(\mathds N)\) by the
spectrum-preserving unitary isometry
\[
        (Uf)(n)=\sqrt{\pi(n)}f(n)
\]
gives the self-adjoint operator \(H=U(-L)U^{-1}\) on \(\ell^2(\mathds N)\) that satisfies
\[
        (Hg)(0)=\lambda g(0)-\sqrt{\lambda\mu}\,g(1),
\]
and, for \(n\ge1\),
\[
        (Hg)(n)
        =
        (\lambda+\mu)g(n)
        -
        \sqrt{\lambda\mu}\,g(n+1)
        -
        \sqrt{\lambda\mu}\,g(n-1).
\]
The vector \((\sqrt{\pi(n)})_{n\ge0}\) is an eigenvector with eigenvalue
\(0\). 
Set
\[
        c:=\sqrt{\lambda\mu}.
\]
Away from the boundary, \(H\) has constant Jacobi coefficients: diagonal
coefficient \(\lambda+\mu\) and off-diagonal coefficient \(-c\). By 
Theorem 3.19 in \cite{TeschlJacobi}
\begin{align} \label{essential spectrum of Jacobi operator}
        \sigma_{\mathrm{ess}}(H)
        =
        [\lambda+\mu-2c,\lambda+\mu+2c].
\end{align}
Therefore, to prove that the spectral gap is 
$\gamma = (\sqrt{\mu} - \sqrt{\lambda})^2 = \lambda + \mu - 2c$,
it remains to check that there are no eigenvalues in the interval
$\bigl(0,\lambda+\mu-2c \bigr).$
Suppose that \(\theta<\lambda+\mu-2c\) is an eigenvalue of \(H\), with
eigenvector \(g\in\ell^2(\mathds N)\). For \(n\ge1\), the eigenvalue
equation is
\[
        (\lambda+\mu-\theta)g(n)
        -
        c\,g(n+1)
        -
        c\,g(n-1)
        =
        0.
\]
This is a system of linear recurrence equations, and the solution is 
determined by $g(0)$ and $g(1)$. With the ansatz $g(n) = r^n$, we get
\begin{align} \label{characteristic recurrence equation}
        c\left(r+\frac1r\right)=\lambda+\mu-\theta.
\end{align}
Since \(\theta<\lambda+\mu-2c\), this equation has one root \(r\in(0,1)\)
and one root \(r^{-1}>1\), and we deduce that all solutions must be of 
the form $g(n) = Ar^n + Br^{-n}$.
The condition \(g\in\ell^2(\mathds N)\) forces \(B=0\), so for \(n\ge1\),
\[
        g(n)=A r^n.
\]
We now use the bulk equation at \(n=1\):
\[
        (\lambda+\mu-\theta)g(1)-c g(2)-c g(0)=0.
\]
Substituting \(g(1)=Ar\), \(g(2)=Ar^2\), and \eqref{characteristic recurrence equation},
we obtain
\[
        c(r+r^{-1})Ar-cAr^2-cg(0)=0,
\]
hence $g(0)=A$.
Now use the boundary equation at \(n=0\):
\[
        \lambda g(0)-c\,g(1)=\theta g(0).
\]
Since \(g(1)=rg(0)\), this gives
\[
        \theta=\lambda-cr.
\]
On the other hand, \eqref{characteristic recurrence equation} gives
\[
        \theta=\lambda+\mu-c\left(r+\frac1r\right).
\]
Comparing the two expressions for \(\theta\), we obtain
\[
        r=\frac{c}{\mu}
        =
        \sqrt{\frac{\lambda}{\mu}}.
\]
Substituting this into \(\theta=\lambda-cr\) yields
\[
        \theta
        =
        \lambda-\sqrt{\lambda\mu}\sqrt{\frac{\lambda}{\mu}}
        =
        0.
\]
Therefore, the only eigenvalue of $H$ below $\lambda+ \mu - 2\sqrt{\lambda \mu}$ is \(0\).
Together with \eqref{essential spectrum of Jacobi operator} and the fact that
$\sigma(H) = \sigma(-L)$, this shows that the spectral gap of $L$ is
\[
        \gamma
        =
        (\sqrt{\mu}-\sqrt{\lambda})^2.
\]

\subsection{Compactness for the hypoelliptic shear on the torus} \label{compactness proof appendix}
We prove here that for the Markov diffusion generator
\[
        L=\partial_y^2+\sin(y)\partial_x
\]
on \(L^2(\mathds T^2)\), the embedding 
\[
(\Dom(L), \|\cdot\| + \|L \cdot \|) \hookrightarrow L^2(\mathds{T}^2)
\]
is compact. For $f \in \Dom(L)$,
write
\[
        f(x,y)=\sum_{m\in\mathds Z}e^{imx}\phi_m(y).
\]
Since the coefficients of \(L\) are independent of \(x\), the \(x\)-Fourier
modes are invariant:
\[
        L(e^{imx}\phi(y))
        =
        e^{imx}L_m\phi(y),
        \qquad
        L_m:=\partial_y^2+im\sin(y).
\]
By Parseval,
\[
        \|f\|^2+\|Lf\|^2
        =
        \sum_{m\in\mathds Z}
        \left(
        \|\phi_m\|_{L^2_y}^2
        +
        \|L_m\phi_m\|_{L^2_y}^2
        \right).
\]
For brevity, set
\[
        E_m
        :=
        \|\phi_m\|_{L^2_y}^2+\|L_m\phi_m\|_{L^2_y}^2.
\]
We have the following estimates, whose proof we defer for now.
\begin{lemma} \label{tail estimate lemma}
    With the above notations, we have for some constant $C$ that
    \begin{align*}
    \|\phi_m\|^2_{H^1_y} & \leq C E_m \\
        \|\phi_m\|_{L^2_y}^2
        &\le
        C|m|^{-1/3}E_m , \quad m \neq 0.
    \end{align*}
\end{lemma}

Let \((f_n)\) be a sequence in $\Dom(L)$ satisfying
\[
        \|f_n\|+\|Lf_n\|\le 1.
\]
Write
\[
        f_n(x,y)=\sum_{m\in\mathds Z}e^{imx}\phi_{n,m}(y),
\]
and define \(E_{n,m}\) as above. Then
\[
        \sum_m E_{n,m}\le C
\]
uniformly in \(n\). The high-$x$-frequency estimate in Lemma \ref{tail estimate lemma}
gives, for \(M\ge1\),
\begin{align} \label{tail estimate 2}
        \sum_{|m|>M}\|\phi_{n,m}\|_{L^2_y}^2
        \le
        C M^{-1/3}\sum_{|m|>M}E_{n,m}
        \le
        C M^{-1/3},
\end{align}
uniformly in \(n\). 
For each fixed \(m\), the estimate
\[
        \|\phi_{n,m}\|_{H^1_y}^2\le C E_{n,m}\le C
\]
shows that \((\phi_{n,m})_n\) is bounded in \(H^1(\mathds T_y)\). By the
Rellich compactness theorem, \(H^1(\mathds T_y)\hookrightarrow
L^2(\mathds T_y)\) compactly. Therefore, by a diagonal argument,
we can extract a subsequence, reindexed by $n$, such that
\((\phi_{n,m})_n\) 
converges in \(L^2_y\) for every $m$. Using the tail 
estimate \eqref{tail estimate 2}, we see that $(f_n)_n$ is Cauchy, hence
convergent in \(L^2(\mathds T^2)\).
Thus, the graph-norm unit ball
\[
        \{f:\|f\|+\|Lf\| \le1\}
\]
is relatively compact in \(L^2(\mathds T^2)\). Equivalently, we have a compact embedding
\[
        (\operatorname{Dom}(L),\|\cdot \|+\|L \cdot\|)\hookrightarrow L^2(\mathds T^2).
\]

\begin{proof}[Proof of Lemma \ref{tail estimate lemma}]
Since smooth functions are dense in $\Dom(L)$, we can assume
\(\phi_m\) to be smooth. Then
\[
        \operatorname{Re}\langle \phi_m,L_m\phi_m\rangle
        =
        \operatorname{Re}\langle \phi_m,\phi_m''\rangle
        =
        -\|\phi_m'\|_{L^2_y}^2.
\]
Therefore,
\[
        \|\phi_m'\|_{L^2_y}^2
        \le
        |\langle \phi_m,L_m\phi_m\rangle|
        \le
        \|\phi_m\|_{L^2_y}\|L_m\phi_m\|_{L^2_y}
        \le
        \frac12E_m.
\]
In particular,
\[
        \|\phi_m\|_{H^1_y}^2\le C E_m.
\]

Next, for \(|m|\ge1\), we estimate \(\sin(y)\phi_m\). Since
\[
        L_m\phi_m=\phi_m''+im\sin(y)\phi_m,
\]
we have
\[
        im\|\sin(y)\phi_m\|_{L^2_y}^2
        =
        \langle \sin(y)\phi_m,L_m\phi_m\rangle
        -
        \langle \sin(y)\phi_m,\phi_m''\rangle .
\]
Taking imaginary parts gives
\[
\begin{aligned}
        |m|\|\sin(y)\phi_m\|_{L^2_y}^2
        &\le
        |\langle \sin(y)\phi_m,L_m\phi_m\rangle|
        +
        |\operatorname{Im}\langle \sin(y)\phi_m,\phi_m''\rangle| .
\end{aligned}
\]
The first term is bounded by
\[
        \|\phi_m\|_{L^2_y}\|L_m\phi_m\|_{L^2_y}
        \le
        \frac12 E_m.
\]
For the second term, integrate by parts:
\[
\begin{aligned}
        \langle \sin(y)\phi_m,\phi_m''\rangle
        &=
        -\int_{\mathds T}\cos(y)\overline{\phi_m(y)}\phi_m'(y)\,dy
        -
        \int_{\mathds T}\sin(y)|\phi_m'(y)|^2\,dy .
\end{aligned}
\]
The second integral is real, and hence
\[
        |\operatorname{Im}\langle \sin(y)\phi_m,\phi_m''\rangle|
        \le
        \|\phi_m\|_{L^2_y}\|\phi_m'\|_{L^2_y}
        \le
        C E_m.
\]
Thus,
\[
        |m|\|\sin(y)\phi_m\|_{L^2_y}^2
        \le
        C E_m.
\]

We now prove that high \(x\)-frequencies have small \(L^2\)-mass. Let
\[
        Z_\delta:=\{y\in\mathds T:|\sin y|\le \delta\}.
\]
For \(0<\delta\le1/2\), \(|Z_\delta|\le C\delta\). Since
\(H^1(\mathds T)\hookrightarrow L^\infty(\mathds T)\) continuously,
\[
        \int_{Z_\delta}|\phi_m(y)|^2\,dy
        \le
        |Z_\delta|\|\phi_m\|_{L^\infty_y}^2
        \le
        C\delta\|\phi_m\|_{H^1_y}^2
        \le
        C\delta E_m.
\]
On the complement \(Z_\delta^c\), we have \(|\sin y|>\delta\), and so
\[
        \int_{Z_\delta^c}|\phi_m(y)|^2\,dy
        \le
        \delta^{-2}\|\sin(y)\phi_m\|_{L^2_y}^2
        \le
        C\delta^{-2}|m|^{-1}E_m.
\]
Choosing \(\delta=|m|^{-1/3}\) gives, for \(|m|\ge1\),
\[
        \|\phi_m\|_{L^2_y}^2
        \le
        C|m|^{-1/3}E_m.
\]
\end{proof}

\end{appendix}

\section*{Acknowledgements}
I am grateful to Sourav Chatterjee for the suggestion to look for extensions of his results, and 
to Amir Dembo for encouraging me in the research process. I thank both of them and Persi Diaconis for helpful feedback.

\nocite{*}

\newgeometry{top=1in, bottom=1in, left=1in, right=1in}

\printbibliography

\end{document}